\pdfoutput=1

\documentclass[11pt]{article}
\pdfmapline{=eufm7     eufm7     <eufm7.pfb}
\pdfmapline{=eufm10     eufm10     <eufm10.pfb}
\usepackage{lmodern}
\usepackage{anyfontsize}
\usepackage{hyperref}
\usepackage[T1]{fontenc}
\usepackage{bbm}
\usepackage[margin=1in]{geometry}
\usepackage{amsthm}
\usepackage{amsmath}
\usepackage{enumerate}
\usepackage{amssymb}
\usepackage[numbers,sort&compress]{natbib}
\usepackage{algorithm, algorithmic}
\usepackage{graphicx}
\usepackage{url}
\usepackage{multirow}

\usepackage{caption}
\usepackage{subcaption}
\usepackage{gensymb,textcomp}
\usepackage{tikz}
\usetikzlibrary{shapes,snakes}
\DeclareMathAlphabet{\mathscr}{OT1}{pzc}{m}{it} 

\newtheorem{theorem}{Theorem}
\newtheorem{corollary}[theorem]{Corollary}
\newtheorem{lemma}[theorem]{Lemma}

\newtheorem{statement}{Statement}
\newtheorem{definition}{Definition}
\newtheorem{modification}{Modification}

\newcommand{\etal}{{et al.~}}
\newenvironment{ntheorem}[1]{\medskip\noindent{\bf Theorem~#1.}\it}{\par}
\newenvironment{nlemma}[1]{\medskip\noindent{\bf Lemma~#1.}\it}{\par}
\newenvironment{ndefinition}[1]{\medskip\noindent{\bf Definition~#1.}\it}{\par}

\newcommand{\argmin}{\operatornamewithlimits{argmin}}

\newcommand{\eat}[1]{}

\newcounter{lp}
\newcommand{\rmax}{r_{m}}

\newcommand{\bin}{\mathop{Bin}}

\newcommand{\ER}{\text{Erd\H{o}s-R\'{e}nyi}}
\DeclareMathOperator{\Var}{Var}

\newcommand{\er}{Erd\H{o}s-R\'{e}nyi}

\newcommand{\E}{\mathbb{E}}

\newcommand{\seed}{\varphi}
\newcommand{\pthreshold}{\Phi}

\newcommand{\diminish}{{\sc Diminish}}
\newcommand{\sequester}{{\sc Sequester}}

\newcommand{\delay}{{\sc Delay}}
\newcommand{\bolster}{{\sc Bolster}}
\newcommand{\bolstera}{{\sc Bolster-A}}
\newcommand{\bolsterb}{{\sc Bolster-B}}

\newcommand{\tc}{{t^*}}
\newcommand{\I}{{\mathcal I}}
\newcommand{\HH}{{\mathcal H}}

\title{Behavioral Intervention and Non-Uniform Bootstrap Percolation}
\date{}
\author{Peter Ballen
\and 
Sudipto Guha\thanks{Department of Computer and Information Sciences, University of Pennsylvania. Email: {\tt \{pballen,sudipto\}@cis.upenn.edu}. Research supported in part by NSF Awards CCF-1117216, IIS 1546151.}}

\begin{document}
\maketitle
\begin{abstract}
Bootstrap percolation is an often used model to study the spread of diseases, rumors, and information on sparse random graphs. The percolation process demonstrates a critical value such that the graph is either almost completely affected or almost completely unaffected based on the initial seed being larger or smaller than the critical value.
In this paper, we consider behavioral interventions, that is, once the
percolation has affected a substantial fraction of the nodes, an
external advisory suggests simple policies to modify behavior (for
example, asking vertices to reduce contact by randomly deleting edges)
in order to stop the spread of false information or disease. We
analyze some natural interventions and show that the interventions
themselves satisfy a similar critical transition.

To analyze intervention strategies we provide the first analytic determination of the critical value for basic bootstrap percolation in random graphs when the vertex thresholds are nonuniform and provide an efficient algorithm. This result also helps solve the
problem of ``Percolation with Coinflips'' when the infection process
is not deterministic -- which has been a criticism about the model. 
We also extend the results to ``clustered'' random
graphs thereby extending the classes of graphs considered. In these graphs the vertices are grouped in a small number of clusters, the clusters model a fixed communication network and the
edge probability is dependent if the vertices are in ``close'' or
``far'' clusters.
We present simulations for both basic percolation and interventions 
that support our theoretical results.

\end{abstract}
\pagenumbering{arabic}

\section{Introduction}
Bootstrap percolation is a model of choice in many contexts modeling spread of information, diseases, etc.

\begin{definition}[Bootstrap percolation] 
Let $G$ be a graph with $n$ vertices and $r:V(G) \rightarrow
\mathbb{N}$ drawn from some family. Given $G$, select $\seed$ vertices 
uniformly at random without replacement and mark them as
infected. Vertex $u$ becomes infected when it has $r(u)$ or more
infected neighbors. $G$ is declared to be infected if $n - o(n)$
vertices are infected. The central question is to determine the
existence and quantify the parameter $\pthreshold$ such that the graph
exhibits a sharp dichotomy. That is, for any fixed $\epsilon>0$, if
$\seed > (1+\epsilon)
\pthreshold$ then $G$ becomes infected with probability vanishingly
close\footnote{For all fixed $\delta>0$, $\exists n(\epsilon,\delta)$
such that $ \forall n \geq n(\epsilon,\delta)$ the graph $G$ with $n$ vertices 
satisfies
the condition with probability at least $1-\delta$.} to $1$, and if
$\seed < (1-\epsilon)
\pthreshold$, $G$ does not become infected with probability vanishingly close to $1$. 
\label{def:bootstrapthresholds}
\end{definition}

An extensive and rich literature exists on the topic of bootstrap
percolation which we discuss shortly.  In this work we focus on
decentralized {\bf behavioral intervention strategies}. Suppose that
the infection is propagating sufficiently slowly. After the
percolation has spread to $\lambda n$ nodes the nodes are instructed
to behave differently (e.g, communicate less, become less susceptible
to new information) which leads to increases of $r(u)$.
Alternatively, the nodes reduce contact -- which corresponds to
dropping edges at random. If the infection has spread reasonably then
the graph already has a ``residual state'' and not all interventions
are useful -- thus we need to quantify how even simple transformations
affect the percolation, and such analysis does not exist in the
literature. The natural questions we ask here are: {\em Does the
required intervention also exhibit a sharp phase transition between
failure and success? Can that region of transition be explicitly
quantified?}

We envision the primary application of large scale intervention to be
useful in the domain of {\bf swarm of particles or agents} which
choose a random network to interact with each other
\cite{swarm}. The intervention in this context arise from the
following: if a large fraction of the swarm is showing undesirable
behavior which is spreading -- what is the effort required to
stabilize such a system? The same question can be asked for
communication networks of machines which often adopt a random topology
for communication and efficiency purposes. In particular we focus on a
hierarchically clustered graph where $n/k$ vertices are in each
cluster and the communication between two nodes in different clusters
is an independent random variable which only depends on the finite
cluster topology defined on the $k$ supernodes. 
We intentionally do not discuss networks with power law distribution
in this paper because in such graphs, the spread shows a phase transition and the graph is infected almost immediately or not at all \cite{powerlaw}. In particular the
infection spreads to all high degree nodes in generation 1, then to a large percentage of the graph in generation 2. Intervention is difficult
to imagine in such a context given such a dramatic change in the number of infected nodes a single
step. Moreover in the context of swarms or social networks for
machines, power law behavior is unlikely to be desirable from the
perspective of communication bottlenecks. 

\paragraph{Challenges and Context.} 
Perhaps unsurprisingly, to prove
sharp dichotomy results for intervention, we need to strengthen and extend existing
results for bootstrap percolation for random graphs to many natural
generations of independent interest. Consider:

\medskip
{\bf (a1) Non-Uniform Thresholds.} The overwhelming majority of the literature 
focuses on uniform constant thresholds, that is, $r(u)=r$ is the same
constant for all vertices. Even for the simplest possible random graph
model, the \ER\ model, classic results such as that of Janson et al
\cite{gnp} (see also \cite{scalia}) only provide bounds for this uniform case. This has to be remedied to provide twosided analysis of interventions -- because at the time 
the intervention happens, there is already residual state (the set of infected vertices). For a healthy vertex $u$ with $3$  infected neighbors, the threshold is now $r(u)-3$. This corresponds to a distributional specification of $r(u)$ which is a natural problem. 

\medskip
{\bf (a2) Small number of early adopters or easily influenced/susceptible  nodes.} 
Moreover the existing literature on bootstrap percolation focuses on
the case where the vertex thresholds are greater than $2$. This is
understandable, because threshold $1$ correspond to a connectivity. In particular for a 
\ER\ graph $G(n,p)$ if $np>1$ then there exists a giant connected component. Therefore 
if the fraction of threshold vertices is $\zeta_1$ and we have
$np\zeta_1>1$ then we will have a giant connected component in the
subgraph induced by the threshold $1$ vertices and percolation will be
instantaneous. However the existing literature does not handle the
complementary and natural regime where $np\zeta_1 \leq 1 -\beta$ for some
$\beta>0$ and $\zeta_1 \ll 1$ --- that is, we have a few
(non-negligible) ``early adopters'' who are influenced as soon as they
are in contact with a new idea or ``easily'' susceptiple individuals
who fall sick at first contact, but the remainder of the vertices
exhibit the key bootstrap percolation property of waiting to see more
evidence of sufficient contact.

\medskip
{\bf (a3) Non-Deterministic Transitions.} The behavior of bootstrap percolation that 
a node deterministically becomes infected when $r(u)$ of its neighbors
are infected have often been criticized. A slightly modified but very
natural model is {\em Percolation with Coin Flips}: an individual node
$u$ becomes susceptible (but not infected) after contact with $s(u)$
infected nodes. Each subsequent contact with an infected node infects
$u$ with probability $z(u)$, say determined by an independent coin
flip. Intuitively node $u$ behaves like having a threshold of $r(u)
\approx s(u)+1/z(u)$ but the transitions are not
deterministic. However as discussed in the example above an expected
threshold of $5r(u)/4$ can be worse than a deterministic threshold of
$r(u)$, and therefore the intuition $r(u) \approx s(u)+1/z(u)$ is not
usable for analysis. No analysis of this natural problem of 
percolation with coinflips exists in the literature to date.

\medskip
{\bf (a4) Hierarchical Networks with Few Levels.} No quantitative analysis of 
sharp dichotomy for bootstrap percolation 
exists for random graphs which are hierarchical in nature -- even for 
hierarchies which are just two levels! While \ER\ Graphs certainly are not often a sufficient model of behavior, hierarchical models can 
model complex phenomenon \cite{hier,hier2}.
Note however that multilevel iterative products lead to power law behavior \cite{kronecker}.

\medskip
We note that there has been studies on stopping the spread of
infection in social networks -- however those strategies have
typically been (i) {\em centralized or before the fact}, i.e., before
the disease starts spreading, see \cite{KDD2014} and references therein;
or (ii) {\em $0/1$ vaccination}, i.e., a node is removed from the
graph or unmodified, see \cite{vaccination} and references
therein. None of those approaches solve (a1)--(a3).
We discuss the result of Janson et al \cite{gnp} (see also
\cite{scalia}) before proceeding further, other related work which are somewhat 
orthogonal to the line of inquiry in this paper is discussed at the
end of the section.  In the $G(n,p)$ notation for \ER\ graphs, an edge
between a pair of vertices is present with probability $p$
(independent of other edges). Under a set of standard assumptions,
such as $pn=\omega(1)$ (a slowly growing function of $n$)
the dichotomy occurs when  
\[ \left(1-\frac1r\right) \left(   \frac{(r-1)!}{np^r}  \right)^{1/(r-1)}\]
vertices are seeded initially. Here  $r(u)=r$ for all $n$ vertices. 
The results on non-uniform
thresholds are minimal. Watts \cite{majorityrule} studied the case on
\ER\ graphs where $r(u) = c \mathop{deg} (u)$ for some $c \in [0,1]$
. Amini \cite{functionofdegree} studied the case on random graphs of a
given degree sequence where $r(u) = g( \mathop{deg}(u))$ and $g$ is a
fixed deterministic function -- however the results in that paper
demonstrate the existence of a sharp dichotomy and explicitly leave
open the computational question. Note however that such arguments
cannot work if $g()$ is changed midway through the percolation as is
the case in interventions. In the context of fixed graphs with $r(u)=r$ for all vertices, Holroyd
\cite{2gridsharp} proved a bound on $a$ that leads to infection on the
2-dimensional grid, which was later improved by Gravner et al
\cite{2gridshaper}. Balogh et al proved a corresponding bound for the
3-dimensional grid
\cite{3gridsharp}, and later proved a general bound for the
$d$-dimensional grid \cite{dgridsharp}. Other results have been found
for hypercubes \cite{hypercube}, tori \cite{torus}, expander graphs
\cite{expanders}, homogeneous trees \cite{homogeneoustree}, 
regular trees \cite{regulartree} and $d$-regular graphs \cite{dregular}. For an arbitrary $G$, approximating
the minimum $\pthreshold$ that leads to infection within small factors is
hard under reasonable complexity assumptions \cite{nphard}. 

\paragraph{Results and Techniques.}
We discuss basic percolation problems (a1)--(a4) in
\ref{sec:contrib-basic}. We discuss interventions next 
in Section~\ref{sec:contrib-intervention} and provide proofs of sharp dichotomy
and consider several simulations in Section~\ref{sec:simulations}

\subsection{Results for Basic Percolation Problems} 
\label{sec:contrib-basic}
We resolve the three scenarios
(a1)--(a4) posited above. In particular, we analyze a {\em Templated
Multisection} graph where the vertex specific thresholds $r(u)$ satisfy
$1 \leq r(u) \leq \rmax$ for some constant $\rmax$. The Templated  Multisection graph is defined
as:

\begin{definition}[{\sc Templated Multisection Graph}]
Let $[k]=\{0,1,\ldots,k-1\}$. Suppose that we are provided a finite template graph $F$ which is a undirected regular graph on $[k]$. The neighborhood of vertex $i \in [k]$ is given by the function $N_F: [k] \rightarrow 2^{[k]}$; where $i \in N_F(j)$ iff $j \in N_F(i)$. Suppose $|N_F(i)| = k_p \leq k$. 
Define $G$ to be a graph with $n$ vertices evenly partitioned into
$k$ clusters of $n/k$ vertices each. Let $\chi(u)$ denote the index of the 
partition $u$ belongs
to. If $\chi(u) \in N_F(\chi(v))$, include edge $(u,v)$ with probability
$p$. Otherwise, include edge $(u,v)$ with probability $q$.
We denote the family of graphs defined in this process as $TM(F,n, k_p, k_q, p, q)$ where $k_q + k_p=k$.
\label{def:tmgraph}
\end{definition}

While some of the results in this paper will extend to clusters of
non-uniform sizes (provided each cluster is large), we omit their
discussion in the interest of brevity.  The $TM(F,n, k_p, k_q, p, q)$
family is illustrated by the following:

\begin{itemize}\parskip=0in
\item \ER\ graphs. This corresponds to a single cluster, $k=1$ and $N_F(0)=\{0\}$. In this case $k_q=q=0$.
\item The Planted Multisection graph is a generalization to $k$ clusters with $N_F(i)=\{i\}$. In this case $k_q=k-1$.
\item Any succinctly described constant degree graph can be used as the template graph 
-- since the intuitive purpose of $F$ is to determine the
communication behavior of nodes in the clusters. Of particular
interest is the ``ring'' type communication where $N_F(i) = \{ i \pm a
\mod{k} \}$ for $|a| \leq \ell$ which defines a ring of $k$ vertices
each node connected to $2\ell$ closest neighbor. We can also
explicitly use any fixed size small world graph.
\end{itemize}

\noindent {\bf Notation:} We use $\eta = n/k$ to denote the number of vertices in a cluster 
and use $\phi = k_p p + k_q q$. The parameters $\eta,\phi$ correspond
to $n,p$ in the \ER\ model. We say $u$ is `near' $v$ if
$\chi(u) \in N_F(\chi(v))$ and $u$ is `far' from $v$ if $\chi(u) \notin N_F(\chi(v))$.
$\bin(x, \lambda)$ denotes a binomial
distribution with $x$ elements and per trial probability of success
$\lambda$. 

\begin{theorem}[Proved in Section~\ref{sec:percolationthresholds}]
\label{thm:mainthm}
Let $\rmax = O(1)$ and fix $\delta,\beta,\epsilon > 0$. 
Let $\{\zeta_r\}_{r=1}^{\rmax}, \sum_{r=1}^{\rmax} \zeta_r = 1$
define a distribution such that $\zeta_1 < 2\zeta_2/3$.
Fix $q\leq p \ll 1/2$.
Given a graph $G$ from the family $TM(F,n,k_p, k_q, p, q)$, with sufficiently many nodes $n \geq n_0(\delta,\beta,\epsilon,k)$ for each $u$, assign $u$ threshold $r$ with
probability $\zeta_r$. 
Let $\phi=p k_p + q k_p$, note $\eta\phi$ is the expected degree.
Define
\begin{eqnarray*}
& & \pi_r(t) =  \Pr[ \bin(k_p t, p) + \bin(k_q t, q) \geq r] \\
& & A(t) = \sum_{r=1}^{\rmax} \zeta_r \pi_r(t) \qquad f(\seed,t) = (n - \seed)A(t) - kt + \seed \\
&& \tc(\seed) = \argmin_{t \leq 1/(3\phi)} f(\seed,t) \\
&& \pthreshold = \min_\seed \{ \seed | \forall \ t \leq 1/(3\phi),  f(\seed,t) \geq 0 \} \qquad \tc=\tc(\pthreshold)
\end{eqnarray*}
Assume (i) $\eta\phi\zeta_1 \leq 1-\beta$, i.e., the expected number of threshold $1$ vertices adjacent to a node is small \footnote{If $\eta\phi\zeta_1 > 1$ then discussion in (a2) applies.}
(ii) $\eta \phi = o(\sqrt{\beta \eta})$, i.e., the graph is not dense otherwise percolation is immediate. Then
\begin{itemize}
\item If $\phi t \leq 1/3$ then
$A(t)$ is convex.
Moreover $\tc \geq \frac{\beta n}{2k(\phi \eta)^2} \rightarrow \infty$ as $n\rightarrow \infty$.

\item Suppose we choose $\seed$ vertices uniformly at random and
set them as infected. If $\seed < (1-\epsilon)\pthreshold$ then $G$ 
does not become becomes infected with probability 
at least $1-O(\epsilon^{-2}/(\tc\beta^{2(\rmax+1)}))$.
If $\seed > (1+\epsilon)\pthreshold$ then an absolute constant fraction of 
the nodes in $G$ become infected with probability at least 
$1-O(\epsilon^{-2}/(\tc\beta^{2(\rmax+1)}))$ (slightly larger constant). Moreover if the 
expected degree $\phi \eta$ is a slowly growing function then with same expression of 
probability close to $1$, the percolation does not stop till $\eta -o(\eta)$ nodes are infected.
\end{itemize}
\end{theorem}

\noindent 
The probabilities of convergence with only absolute constants in the $O()$ all evaluate to
 $1-O\left(\frac{\rmax^{O(\rmax)}k(\eta\phi)^2}{\epsilon^2\beta^{2\rmax+3}n}\right)$ and is (inverse) polynomially close to $1$ when the expected degree $\eta\phi$ is $o(\sqrt{n})$. 
Theorem~\ref{thm:mainthm} follows the argument template of Janson
\etal \cite{gnp}, but differs significantly in the internal analysis. In case of 
uniform thresholds and a single cluster, it was sufficient to
approximate the Binomial by Poisson in defining $A(t)$. However the
Poisson apprixation in $A(t)$ and a blind application of \cite{gnp}
does not provide us the desired result because {\em now the
approximations have to commensurate with the different thresholds simultaneously}.
At the same time the heart of the proof in \cite{gnp} relies on the construction of a
Martingale and a reverse Martingale for a fixed uniform threshold
$r(u)=r$ for an \ER\ graph. We show that we can construct similar
martingales as the value of $r$ is varied, even as the graph has
multiple clusters. Martingales are invariant under addition and thus
the key is to bound their step sizes. However the addition of martingales 
for different thresholds is manageable only with more precise approximations of $A(t)$.
The changes in the internal analysis reflects this key difference. However the 
surprising insight of the overall proof is that {\em even though the percolation is
nonlinear in the connectivity parameter, $\pthreshold$ is linear in
the distribution parameters, and the effects are separable!}

The next result is a consequence of the separability and
distributional result proven in
Theorem~\ref{thm:mainthm}. The basic intuition 
is that the coins can be ``preflipped'' ahead of time to
reduce percolation with coinflips to a {\em distribution} over
percolation with non-uniform thresholds.

\begin{corollary}[{\sc Percolation with Coinflips}]
\label{cor:coinflips}
For $G=TM(F,n,k_p, k_q, p, q)$, suppose the distribution of thresholds
of is chosen as follows: a vertex becomes susceptible after $s(u)$ neighbors of $u$ have been infected. Subsequent to
becoming susceptible, a vertex $u$ becomes infected with probability
$z(u)>\delta>0$ as soon as a new neighbor becomes infected where $\delta \in (0,1]$ is a constant or when $\rmax$ neighbors are infected.
Given this setting
we can determine the percolation threshold $\pthreshold$ explicitly, even for non-uniform $z(u)$.
Observe that if $s(u)\geq 1$ for a large fraction of the nodes then the condition $\zeta_1 \leq 2\zeta_2/3$ holds, e.g., if $s(u)\geq 1$ for $2/3$ of the nodes then $\zeta_1 \leq 2\zeta_2/3$ if all $z(u) \leq 1/2$. 
\end{corollary}

\paragraph{Summary:} We ran several simulations to verify the application of 
Theorem~\ref{thm:mainthm} in the context of multiple graphs which are
``small network of clusters''. We focused on graphs with $n=10000$
nodes and rings with $10,20$ clusters, the $3$-D cube with $8$
clusters alongside the standard $G(n,p)$ graph. We varied the
distribution of the threshold in simple ways so that the results can
be verified conceptually.  The theorem and the simulations agreed and
these are presented in Section~\ref{sec:simulations}. The network did
not affect the thresholds significantly, but the expected degree and
the distribution of the vertex thresholds had a strong impact as predicted. We now focus on intervention strategies.

\subsection{Results for Interventions} 
\label{sec:contrib-intervention}
We now discuss percolation problems when we have a chance to modify
the behavior of edges or vertices. The classic example is that after
the infection has spread to $\lambda n$ individuals, better health
practices are announced, i.e, which reduce contact (edges) or increase
the thresholds of susceptibility and infection for a vertex.  We
consider the following interventions:
\begin{itemize}\parskip=0in
\item \bolster.  This corresponds to assigned every vertex with threshold $r$ a new threshold $r'$ from a distribution $\zeta'(r)$. Note that $\zeta'(r_1)$ may be different from $\zeta'(r_2)$.
\item \delay. This corresponds to modifying $z(u)$ to $z'(u)$ in the coinflips model for the remainder of the percolation. This is a special case of \bolster\ .
\item \sequester. This corresponds to dropping the edges between healthy and the 
infected vertices independently with probability $1-\alpha_p$ for the edges corresponding 
to neighborhoods in $F$ and $1-\alpha_q$ for the other edges. The edges remain dropped throughout the percolation. Edges connecting two healthy vertices are unaffected.
\item \diminish. This corresponds to permanently 
dropping the edges between all vertices 
in the graph with probability $1-\alpha_p$ for the edges corresponding 
to neighborhoods in $F$ and $1-\alpha_q$ for the other edges.
\end{itemize}

\noindent 
Note that the standard strategy of {\em vaccination} corresponds to
setting $r(u)= \infty$ for a healthy node $u$ or equivalently,
removing that node $u$. There is a large literature on this removal
behavior, see \cite{vaccination} and references therein. Edge removal strategies have
been considered in the literature, see \cite{KDD2014} and references
therein. However none of the existing strategies can express the adaptive nature of the four interventions we consider.

\noindent

We assume that the original graph $G$ is chosen so that the assumptios in Theorem~\ref{thm:mainthm} are satisfied ($G$ is sufficiently sparse) and that $G$ has no threshold-1 vertices. We define $\tau$ to be the generation at which the intervention is applied. Let $\I(\tau)$ be the set of infected vertices at generation $\tau$ and define $\I(\tau-1)$ similarly. Let $\HH(r)$ be the set of threshold-r healthy vertices. This is all the data we need to determine whether the intervention is successful.

\begin{theorem}[Proved in Section~\ref{sec:intervention}]
\label{thm:interventionsuperthm}
Assume $n,p,q,r,\lambda$, $\I(\tau)$, $\I(\tau-1)$, $\HH(r)$ are known and that $\I(\tau) < k/(3 \phi)$. Given either $\zeta(r)$ (for \bolster ), $z'$ (for \delay ), or $\alpha_p$ and $\alpha_q$ (for \diminish\ and \sequester ), it is possible to determine whether the intervention is successful and $G$ becomes infected with probability $1-1/poly(n)$.
\end{theorem}

When $\I(\tau) > k/(3 \phi)$, $\I(\tau+1)$ consists of a positive fraction of the nodes. In other words, the intervention has occured too late and there are too many infected nodes to do a meaningful analysis.
The key idea of the proof is that knowing $\I(\tau),  \I(\tau-1), \HH(r)$, we define $\HH_a$ to be the set of infected vertices with $a$ infected neighbors. $\Pr[u \in \HH_a]$ can be explicitly calculated. Using this information, we create a new TM graph $J$ with $|\HH|$ vertices and the same cluster structure as $G$. For every vertex $u \in \HH_a$, we add a vertex $v$ with threshold $r'(u)-a$ to $J$.
In this way, the probability $J$ becomes infected is equal to the probability $G$ becomes infected, and we can use Theorem~\ref{thm:mainthm}.

Less formally, when the intervention is applies to $G$, every vertex has a ``residual'' state which we can estimate. Interestingly, we can estimate this information knowing only $\I(\tau)$ and $\I(\tau-1)$; we do not need to know any information about the early generations, this corresponds to the memoryless martingale behavior of the basic percolation processes.

This construction also showcases the need to resolve (a1) and (a2). Even if $G$ has uniform thresholds, $J$ will have non-uniform thresholds and $J$ will have a small number of easily influenced threshold-1 vertices.

\subsection{Simulations}
\label{sec:simulations}

To test Theorem~\ref{thm:mainthm} and
Corollary~\ref{cor:coinflips} we performed several different
simulations.  The parameter $\epsilon=0.1$ and for each setting we
repeated the experiments for $50$ graphs and for each graph the
experiment was repeated $50$ times. We use the color {\bf green} to
denote the cases where the percolation stopped (graph was mostly healthy) and color {\bf red}
when the infection spread exceeded $90\%$ of the nodes. The number of vertices was
always $10000$. We also varied $k=\{1,10,20\}$. For $k>1$ we chose the
ring topology with $\ell=1$ (which corresponds to $k_p=3$).
Figure~\ref{fig1} consider nonuniform thresholds and for
$k=1$ we chose $p=10/n$. For $k=10$ we set $p=50/(3n)$ and
$q=50/(7n)$. For $k=20$, we chose $p=100/(3n)$ and $q=100/(17n)$. Note
that for all cases $pk_p\eta=5$ and $qk_q\eta=5$ where $\eta=n/k$, that is the average
degree is $5$ within the cluster and $5$ outside the cluster. 
We repeat the same settings as in Figure~\ref{fig1} for percolation with coinflips and the results are in Figure~\ref{fig-coin}. The case for $k=10$ was close to the two 
extreme cases shown. Again the results are parallel those in
Figure~\ref{fig1} even though $\rmax=20$.

\begin{figure}[H]
\begin{subfigure}[t]{0.3\textwidth}
\centering
\includegraphics[scale=0.4]{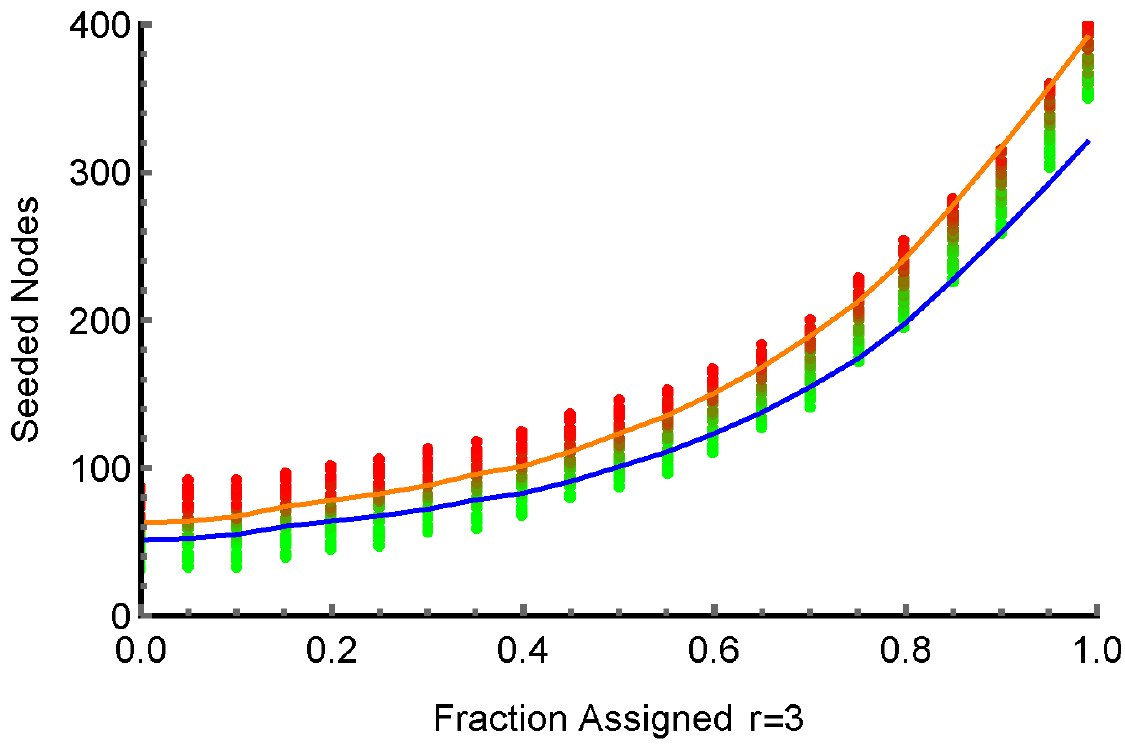}
\subcaption{\er\ $k=1$ case\label{fig1a}}
\end{subfigure}
\begin{subfigure}[t]{0.3\textwidth}
\centering
\includegraphics[scale=0.4]{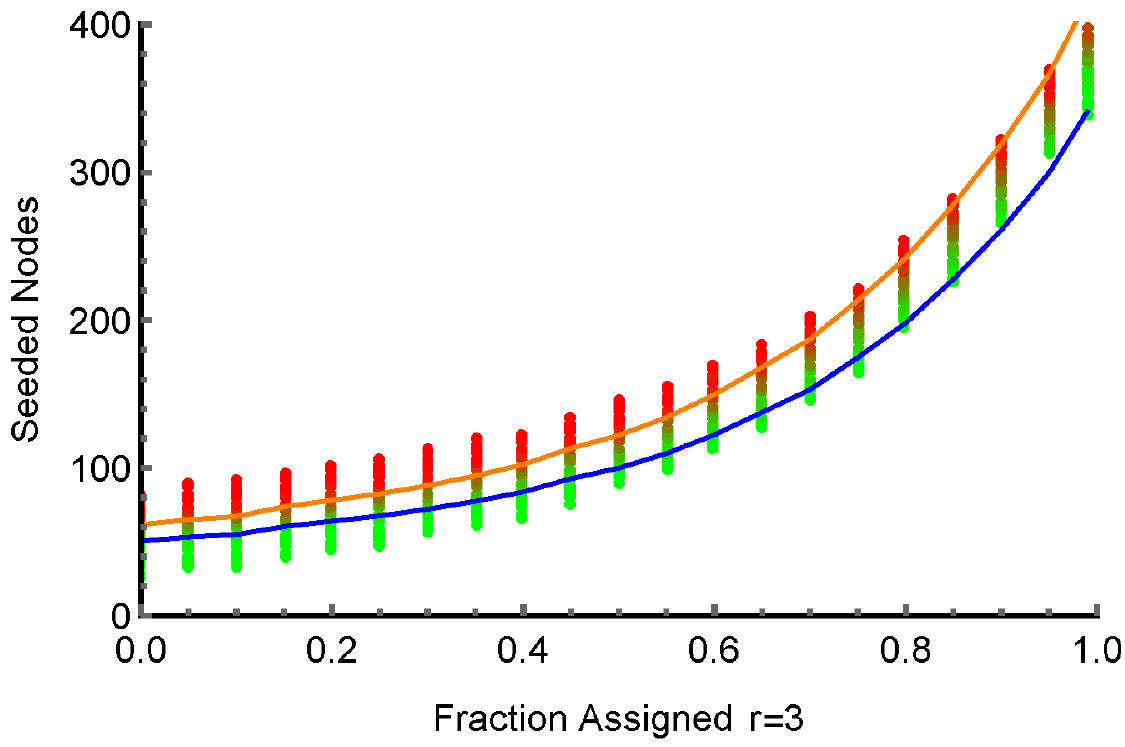}
\subcaption{Ring of $k=10$ clusters \label{fig1b}}
\end{subfigure}
\begin{subfigure}[t]{0.3\textwidth}
\centering
\includegraphics[scale=0.4]{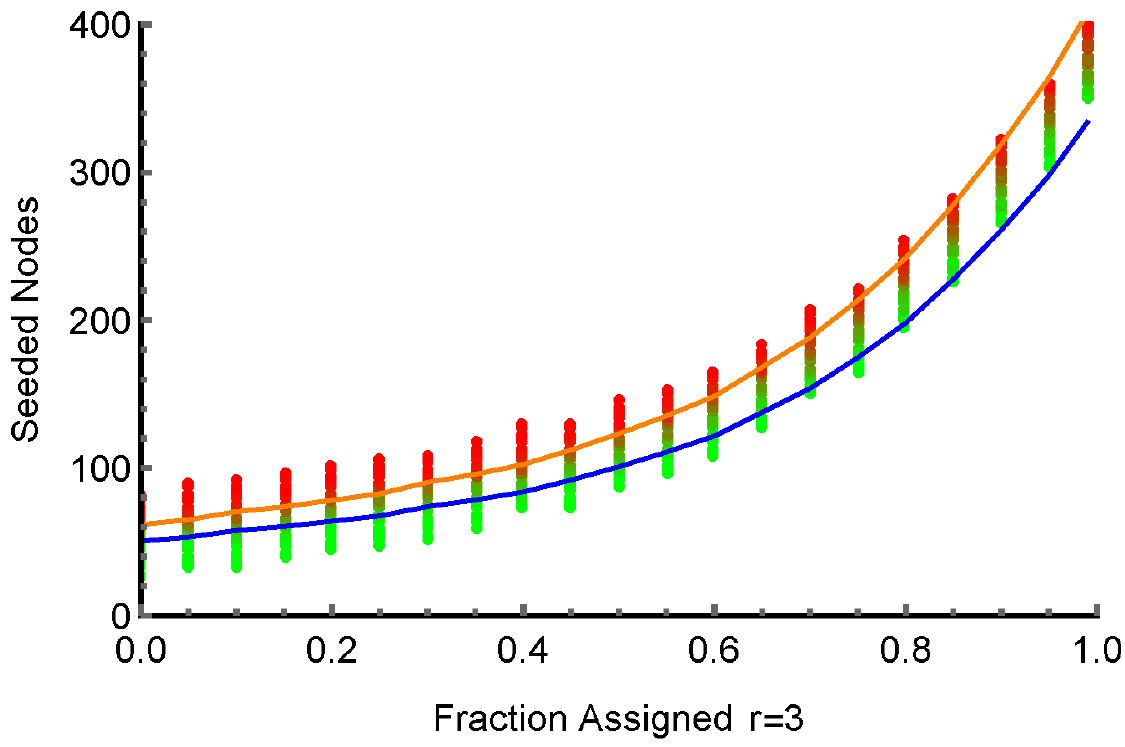}
\subcaption{Ring of $k=20$ clusters \label{fig1c}}
\end{subfigure}
\caption{\label{fig1} 
Non Uniform Thresholds: the $x$ axis indicates the fraction of
vertices which have threshold $3$ the remainder have threshold
$2$. The two lines correspond to
$(1-\epsilon)\pthreshold,(1+\epsilon)\pthreshold$ when the threshold
was $\pthreshold$. Each of the points correspond to the average of
$50$ experiments and the color depends on the fraction of times the
percolation succeeded (red) or stopped (green). For the chosen
settings Theorem~\ref{thm:mainthm} predicts that the thresholds would
be the same for $k=10$ and $k=20$. 
The $k=1$ case is different but the underlying random
variables are close in distribution.}
\end{figure}

\begin{figure}[H]
\begin{subfigure}[t]{0.5\textwidth}
\centering
\includegraphics[scale=0.45]{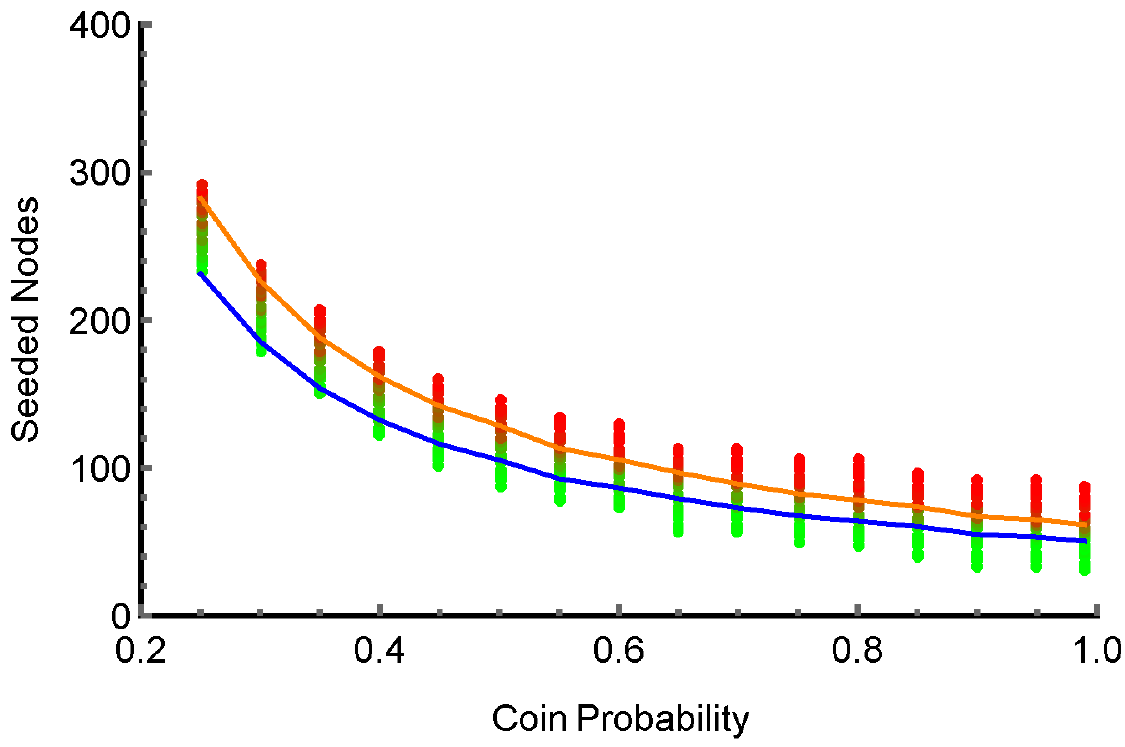}
\subcaption{$k=1$ case\label{fig-coina}}
\end{subfigure}
\begin{subfigure}[t]{0.5\textwidth}
\centering
\includegraphics[scale=0.45]{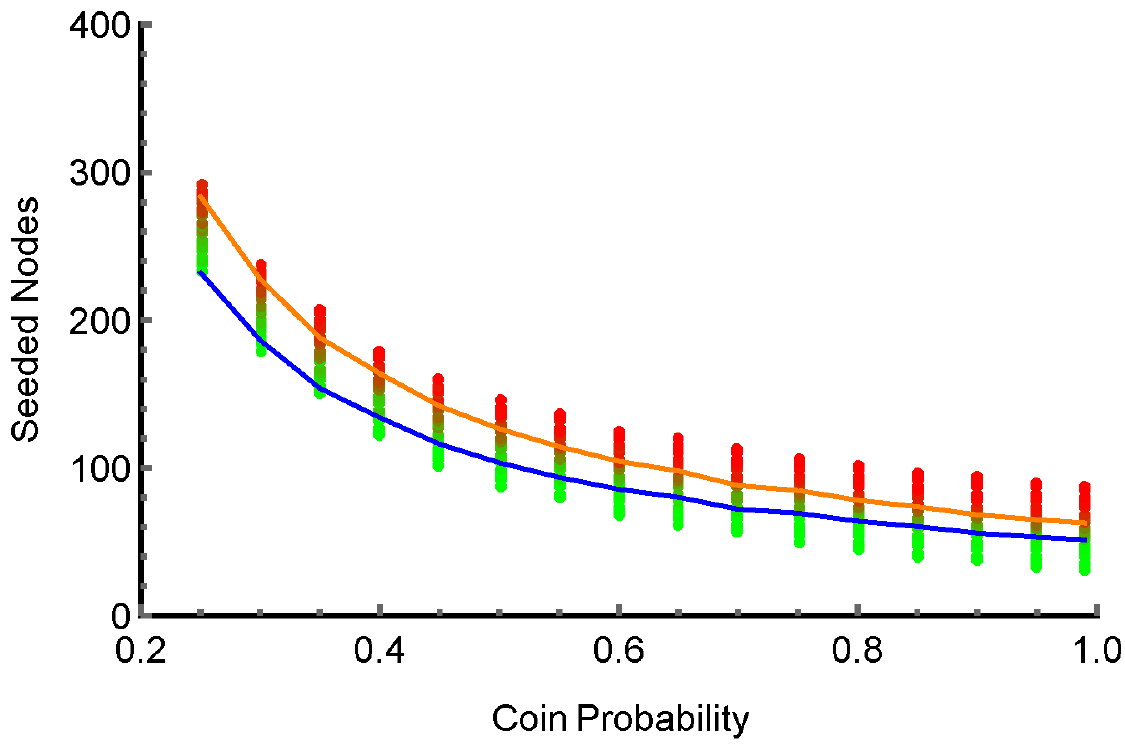}
\subcaption{Ring of $k=20$ clusters \label{fig-coinc}}
\end{subfigure}
\caption{\label{fig-coin}
Percolation with coinflips: For interpretability we considered
$s(u)=1$ and $z(u)=z$ for all vertices. We do not show the $k=10$ case which is similar to the above two. $\rmax$ was set to $20$. The
$x$ axis indicates the coin probability. 
}
\end{figure}

We consider unbalanced setups where the expected
degree of a node within a cluster is different from the expected
degree of the node outside, i.e., $pk_p \neq qk_q$, in
Figure~\ref{fig3}. In Figure~\ref{fig-cube} we consider the clusters
arranged as the vertices of a $3$-D cube where we have $k=8$ clusters
and $k_p=4$. In all cases the result is consistent with the prediction of Theorem~\ref{thm:mainthm}. 
Although the network structure did not affect the critical values and the percolation, as long as the expected degree was the same, changing the expected degree had a much greater impact. In Figure~\ref{fig2} we change the average degree parameter (and the threshold to be between $3$ and $4$)  --- as predicted, the effect is clearly seen on the critical value of percolation.

\begin{figure}[H]
\begin{subfigure}[b]{0.5\textwidth}
\centering
\includegraphics[scale=0.45]{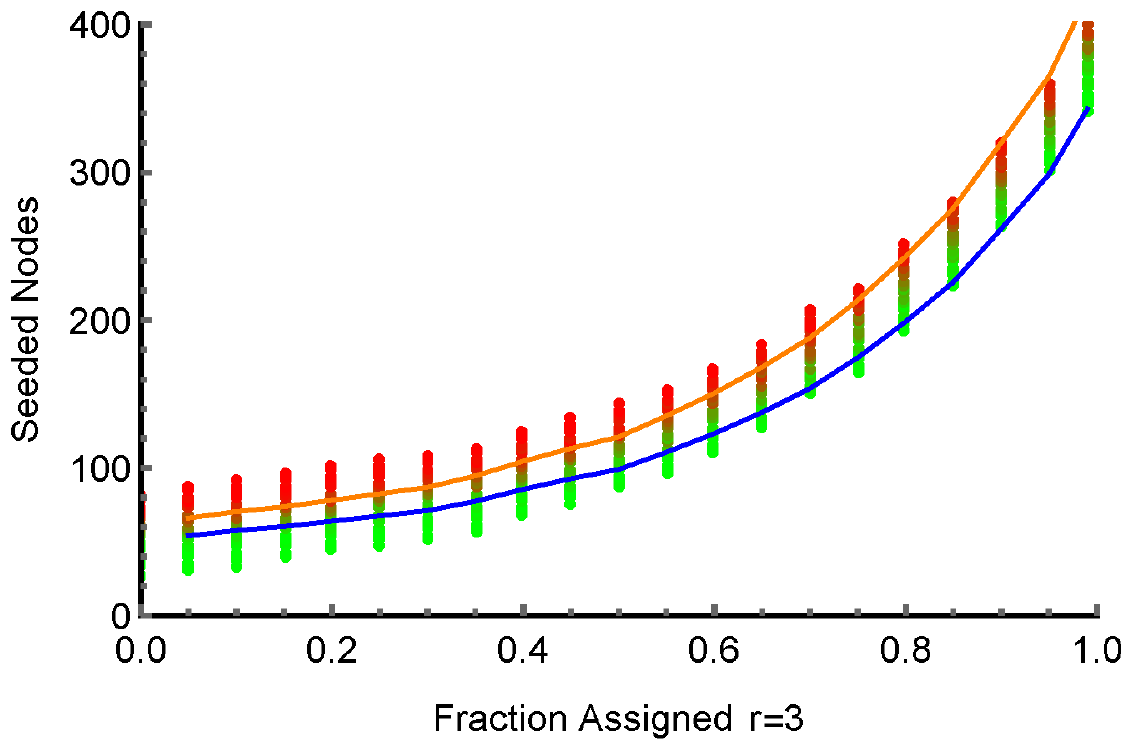}
\subcaption{\label{fig3} Ring of $k=10$ clusters, $p=70/(3\eta),q=30/(7\eta)$}
\end{subfigure}
\begin{subfigure}[b]{0.5\textwidth}
\centering
\includegraphics[scale=0.45]{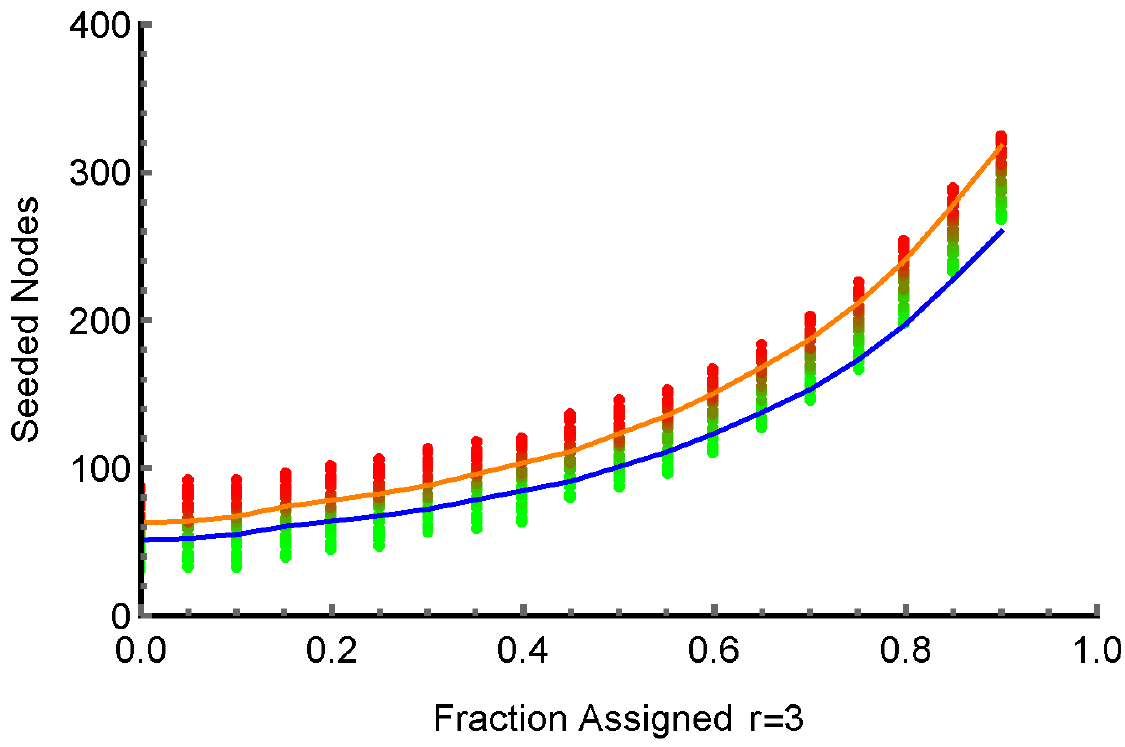}
\subcaption{\label{fig-cube} A cube connection network $k=8, k_p=4$.}
\end{subfigure}
\caption{Different topologies of graphs, Average degree is $10$ in both cases.
In \eqref{fig3} $p=70/(3n),q=30/(7n)$, whereas in \eqref{fig-cube} $p=15/n$ and $q=5/n$. }
\end{figure}

\begin{figure}[H]
\begin{subfigure}[t]{0.5\textwidth}
\centering
\includegraphics[scale=0.45]{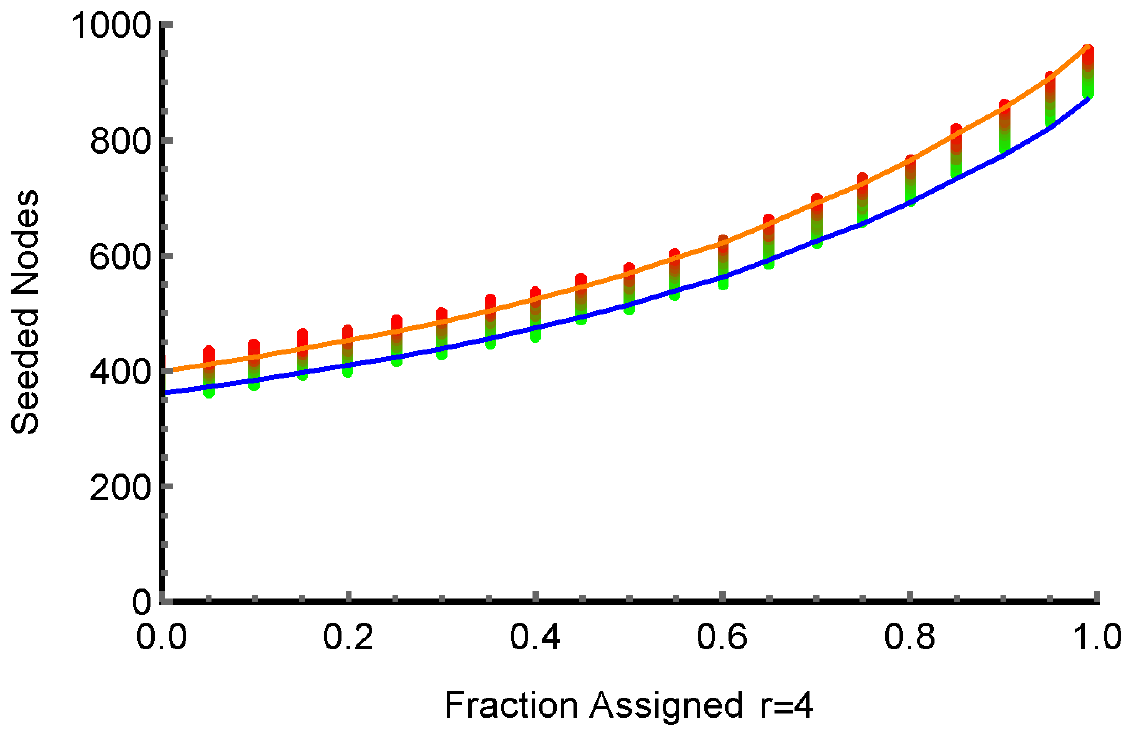}
\subcaption{Expected degree $10$\label{fig2a}}
\end{subfigure}
\begin{subfigure}[t]{0.5\textwidth}
\centering
\includegraphics[scale=0.45]{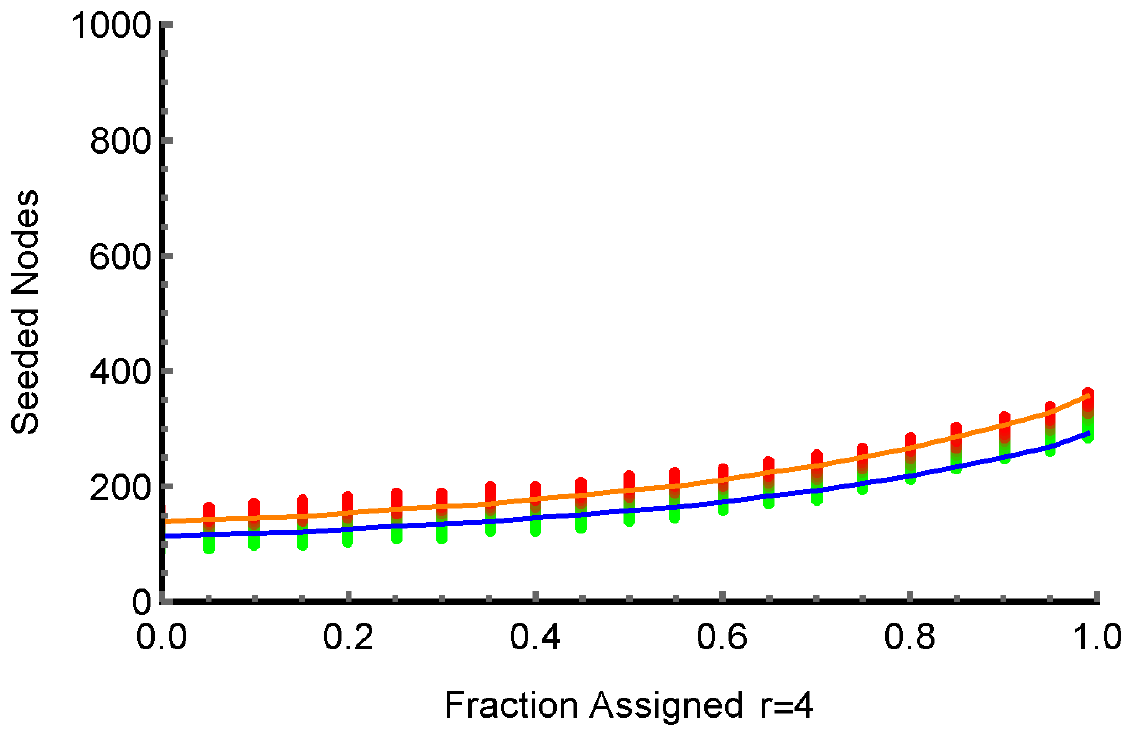}
\subcaption{\label{fig2b} Expected degree $20$}
\end{subfigure}
\caption{\label{fig2} Ring of $k=10$ clusters, thresholds between $3$ and $4$: we vary the total (expected) degree of the nodes. We kept the $y$-axis scale the same for comparison.}
\end{figure}

\paragraph{Interventions.} Figure~\ref{fig:bolsterk1} and Figure~\ref{fig:bolsterk10} show the results of \bolster\ . In both simulations, $n=10000$, $\lambda = .1$, $k=1$, $r=2$ uniformly. Figure~\ref{fig:bolsterk1} depicts an \er\ graph with $k=1$, $p=7/n$. Figure~\ref{fig:bolsterk10} depicts a ring graph with $k=10$, $p=4/3000$, $q=3/7000$. In both figures, green dots imply percolation stopped and red dots imply the spread was complete. The blue dots and line correspond to $1-\epsilon$ times the expected cutoff point, where $\epsilon=.1$. The black $x$s and line correspond to $1+\epsilon$ times the estimated cutoff.

\begin{figure}[H]
\begin{subfigure}[t]{0.5\textwidth}
\centering
\includegraphics[scale=0.45]{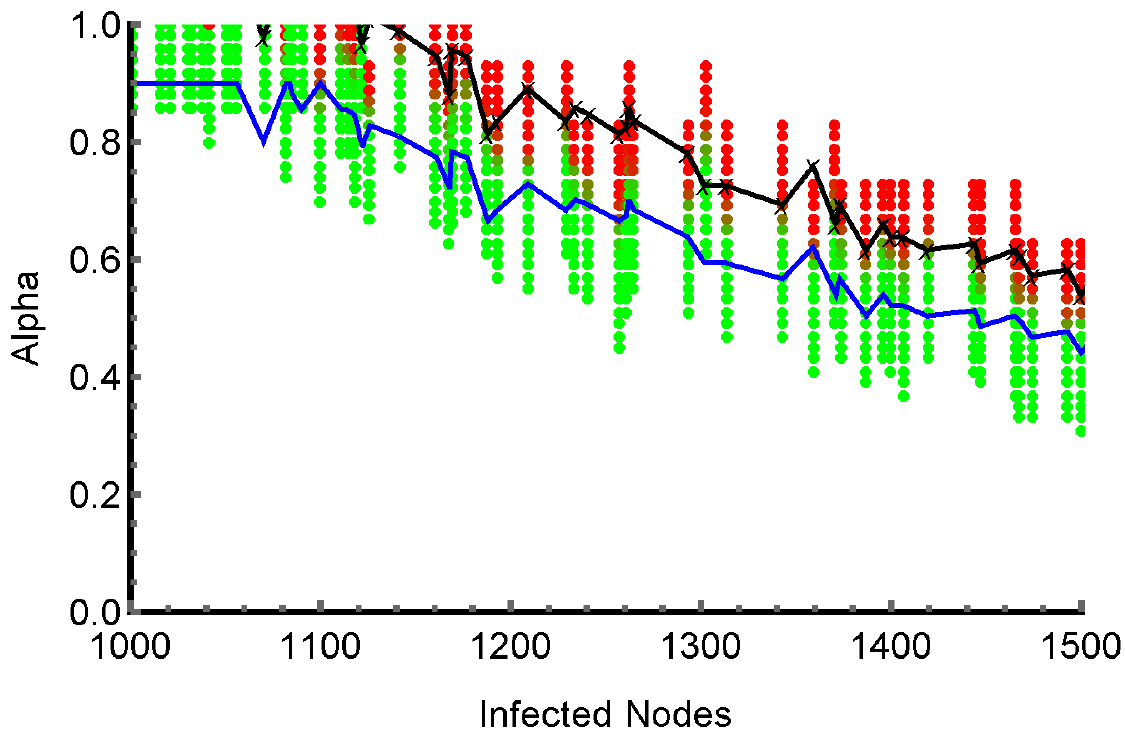}
\subcaption{\label{fig:bolsterak1}Bolster Strategy A, $k=1$}
\end{subfigure}
\begin{subfigure}[t]{0.5\textwidth}
\centering
\includegraphics[scale=0.45]{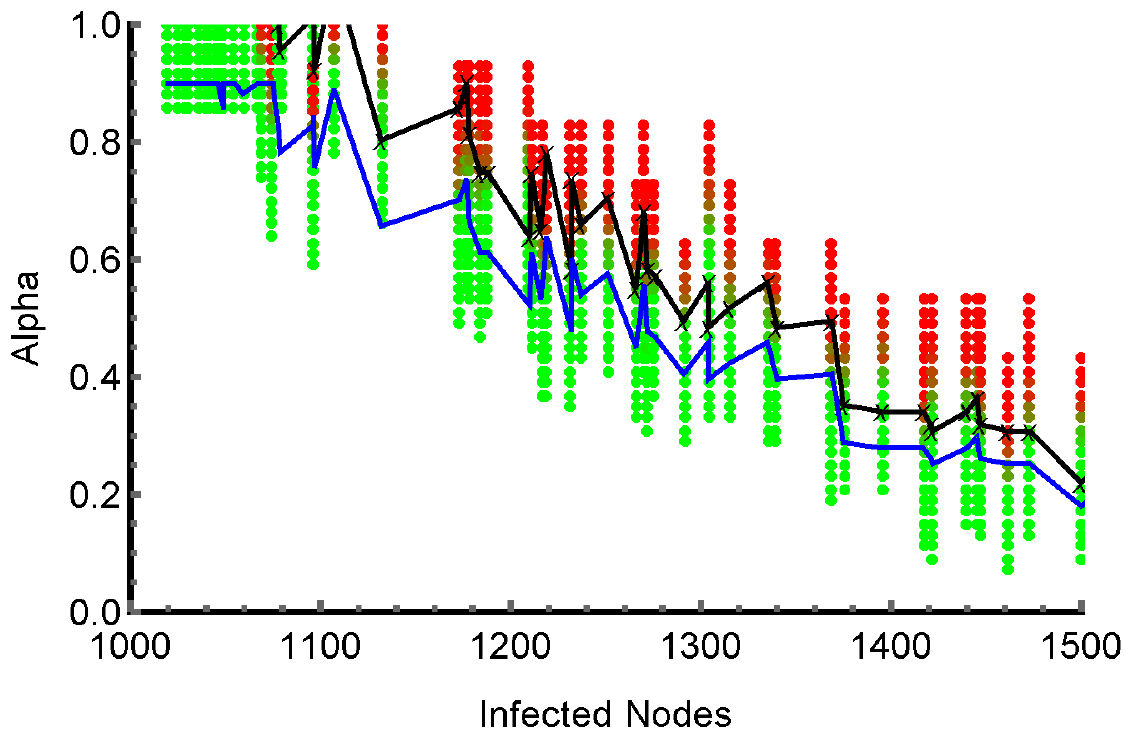}
\subcaption{\label{fig:bolsterbk1} Bolster Strategy B, $k=1$}
\end{subfigure}
\caption{\label{fig:bolsterk1} Graph is \er\ .  The $x$-axis corresponds to $\I(\tau)$ plus the number of just infected vertices. Each vertical line corresponds to a graph (50 such graphs) and for each value of $\alpha$ we show the average of $50$ trials. The blue dots correspond to the estimated lower bound with $\epsilon=.1$ and the $x$ corresponds to the estimated upper bound. Strategy A seems to be better amenable to estimation. Green dots imply that the percolation stopped and red implies that the spread was complete.}
\end{figure}

\begin{figure}[H]
\begin{subfigure}[t]{0.5\textwidth}
\centering
\includegraphics[scale=0.45]{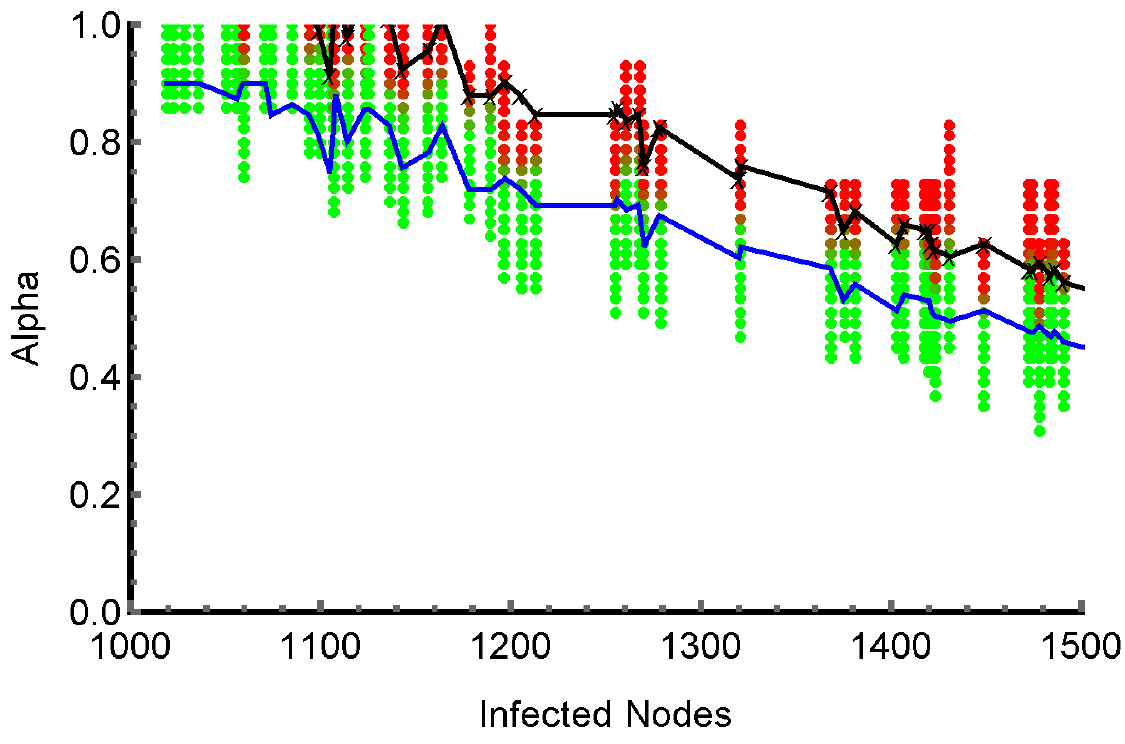}
\subcaption{\label{fig:bolsterak10}Bolster Strategy A, $k=10$}
\end{subfigure}
\begin{subfigure}[t]{0.5\textwidth}
\centering
\includegraphics[scale=0.45]{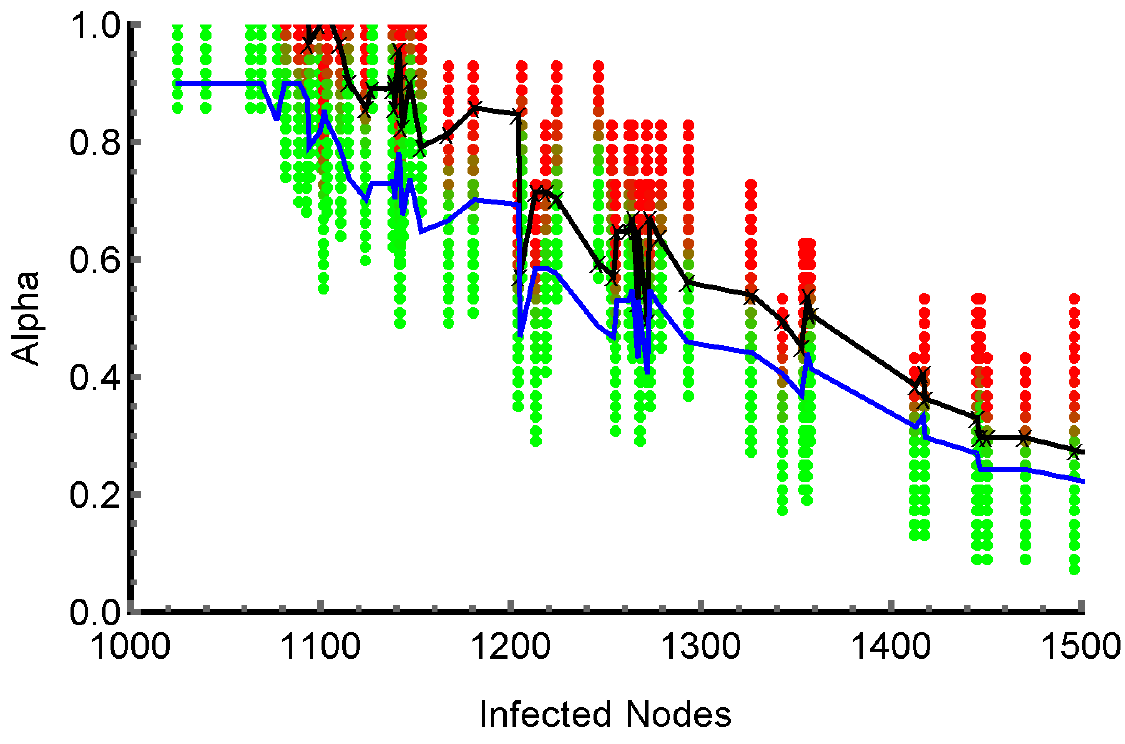}
\subcaption{\label{fig:bolsterbk10} Bolster Strategy B, $k=10$}
\end{subfigure}
\caption{\label{fig:bolsterk10} Graph is a ring of $k=10$ clusters. The $x$-axis and vertical lines have the same meaning as in Figure~\ref{fig:bolsterk1}.}
\end{figure}

We consider two possible \bolster\ interventions. For \bolstera\ (Figure~\ref{fig:bolsterak1} and \ref{fig:bolsterak10}), with probability $\alpha$ we increase the threshold by $1$ (to $3$) and with probability $1-\alpha$ we increase the threshold by $2$ (to $4$). For \bolsterb\ (Figure~\ref{fig:bolsterbk1} and \ref{fig:bolsterbk10}), with probability $1/2 + \alpha/2$ we increase the threshold by $1$ (to $3$) and with probability $1/2 - \alpha/2$ we increase the threshold by $2$ (to $5$). Note that for both interventions, the expected post-intervention threshold is $4-\alpha$ and that $\alpha=0$ corresponds to the strongest possible intervention.

The first interesting result is that \bolstera\  is substantially more powerful than \bolsterb\  (the intervention is successful with a higher value of $\alpha$). For example, having $50\%$ threshold-3 vertices and $50\%$ threshold-5 vertices is worse than having $100\%$ threshold-4 vertices; the vulnerable threshold-3 vertices become infected and then they infect the threshold-5 vertices.

The second interesting result is that for two graphs $G_1$ and $G_2$, there are times when $G_1$ has more infected nodes than $G_2$ but it is easier to stop the infection on $G_1$ than in $G_2$. This is because there are two factors that determine the effectiveness of the intervention: $\I(\tau)$ and $\I(\tau)-\I(\tau-1)$. Thus, instead of having a two-dimensional decision boundary, we have a three-dimensional boundary. We illustrate this boundary in Figure~\ref{fig:bolster3d}, which is the three-dimensional version of Figure~\ref{fig:bolsterak1} (\bolstera\ and $k=1$).

\begin{figure}[H]
\begin{subfigure}[t]{0.33\textwidth}
\centering
\includegraphics[scale=0.5]{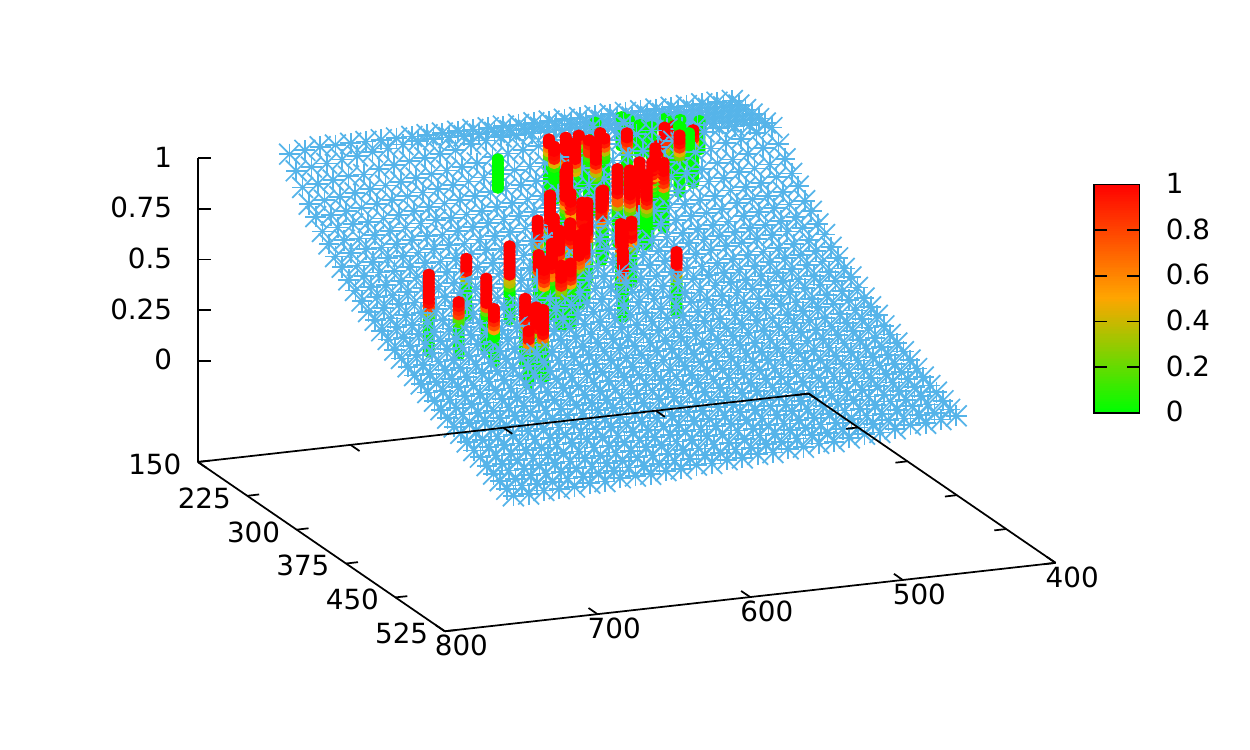}
\end{subfigure}
\begin{subfigure}[t]{0.33\textwidth}
\centering
\includegraphics[scale=0.5]{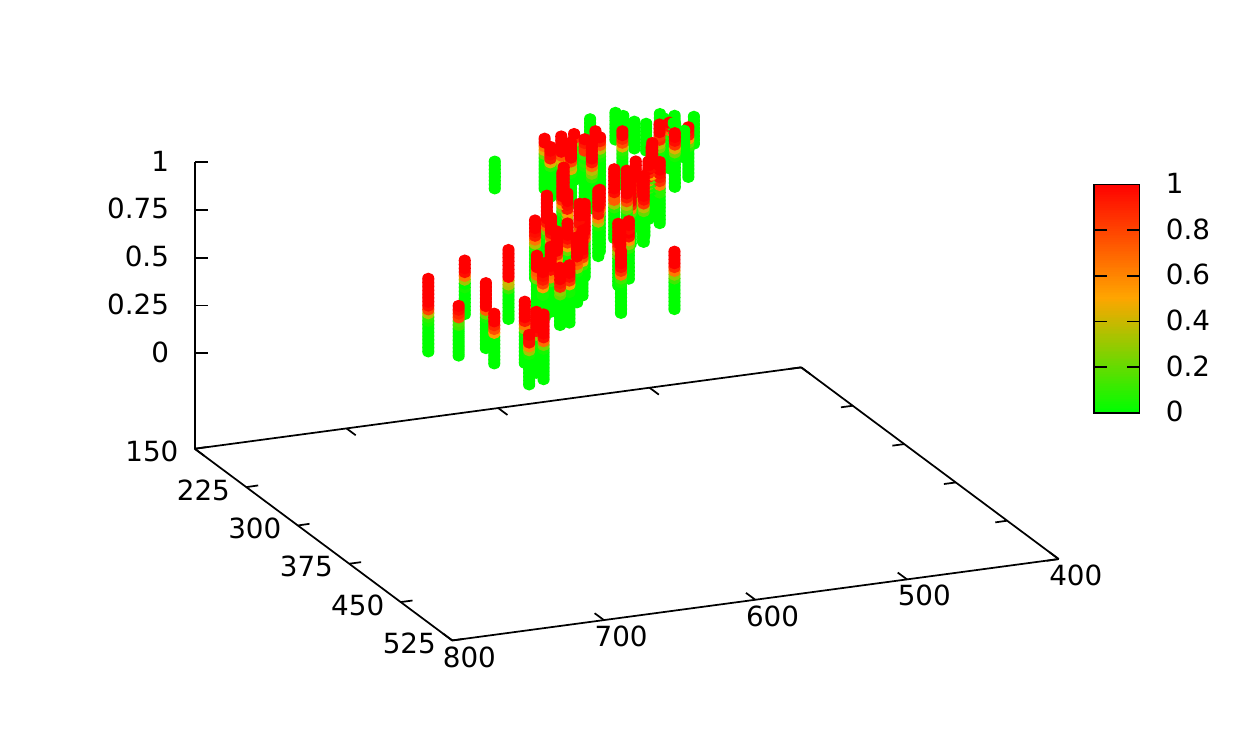}
\end{subfigure}
\begin{subfigure}[t]{0.33\textwidth}
\centering
\includegraphics[scale=0.5]{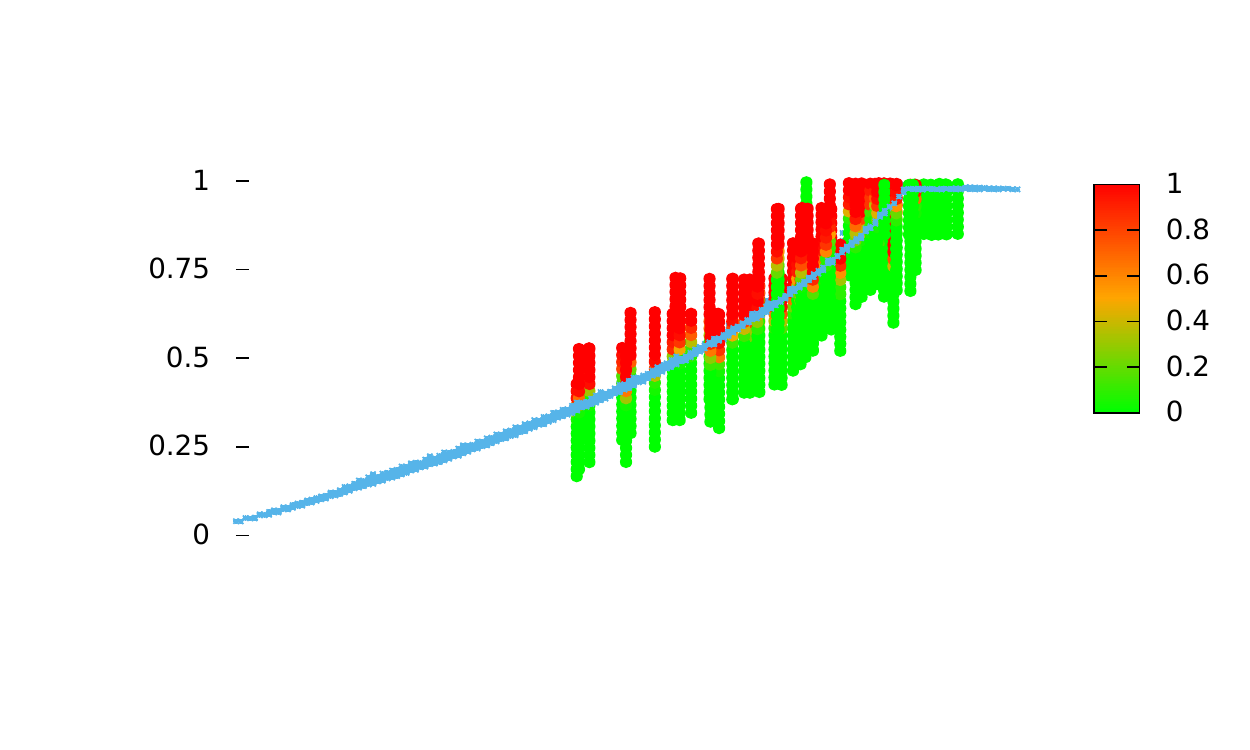}
\end{subfigure}
\caption{\label{fig:bolster3d} The three axes correspond to $\I(\tau)$ (ranging from $400$--$800$), $\I(\tau)-\I(\tau-1)$, and $\alpha$. Green dots imply percolation stopped and red dots imply the spread was complete. The blue surface corresponds to the theoretical transition surface, excluding the $1 \pm \epsilon$ factor for visual clarity. The perspective is chosen to show that the surface occludes the green points in the first two plots. The third plot shows that the surface separates the colors (with a few outliers, note we are showing the $\epsilon=0$ surface).}
\end{figure}

The final interesting result is that \bolsterb\ is substantially nosier than \bolstera\ . Pulling apart a single simulation reveals why. In Figure~\ref{fig:bolstercomparison}, we zoom in on a single graph and compare various hypothetical interventions. The x-axis corresponds to the value of $\alpha$, and black line corresponds to the hypothetical value of $\I(\tau)$ that would lead to the spread of infection given this value of $\alpha$. The actual value of $\I(\tau)$ along with simulated results is dpecited by the red-green line. When the black theoretical value is greater than the actual $I(\tau)$ , we expect the percolation to stop (and the result-line should be green). When the black theoretical value is less than $\I(\tau)$, we expect the percolation to spread (and the result-line should be red).

\begin{figure}[H]
\begin{subfigure}[t]{0.5\textwidth}
\centering
\includegraphics[scale=0.45]{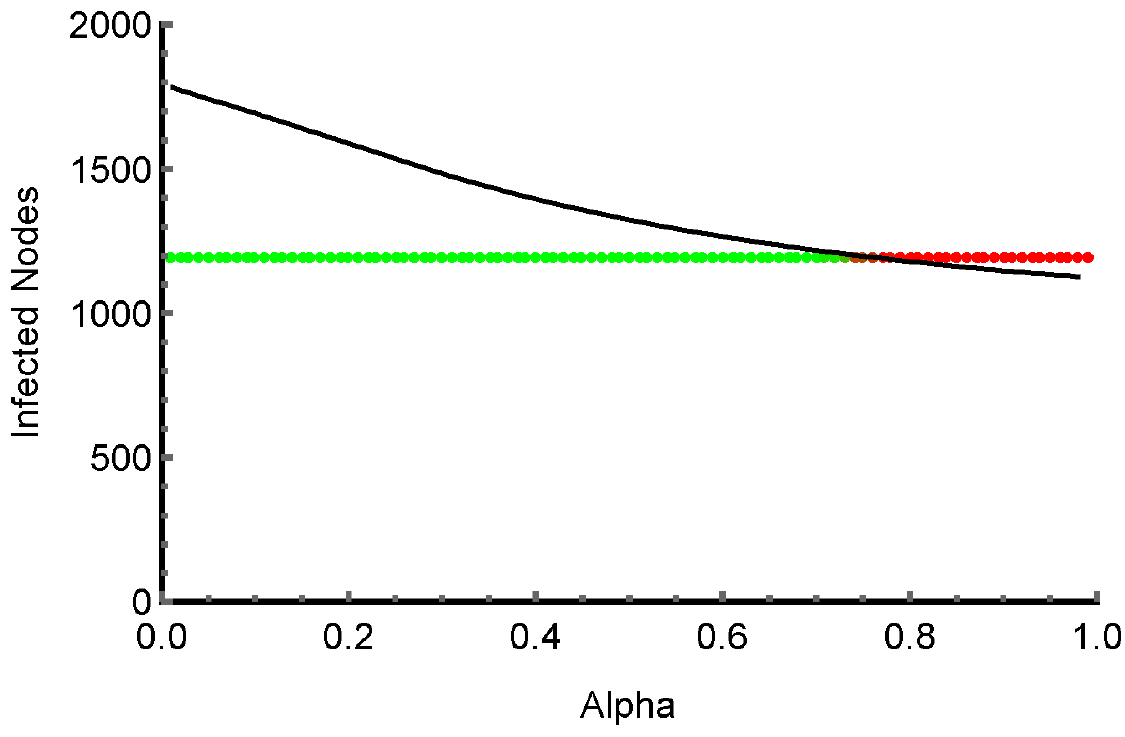}
\subcaption{Bolster Strategy A}
\end{subfigure}
\begin{subfigure}[t]{0.5\textwidth}
\centering
\includegraphics[scale=0.45]{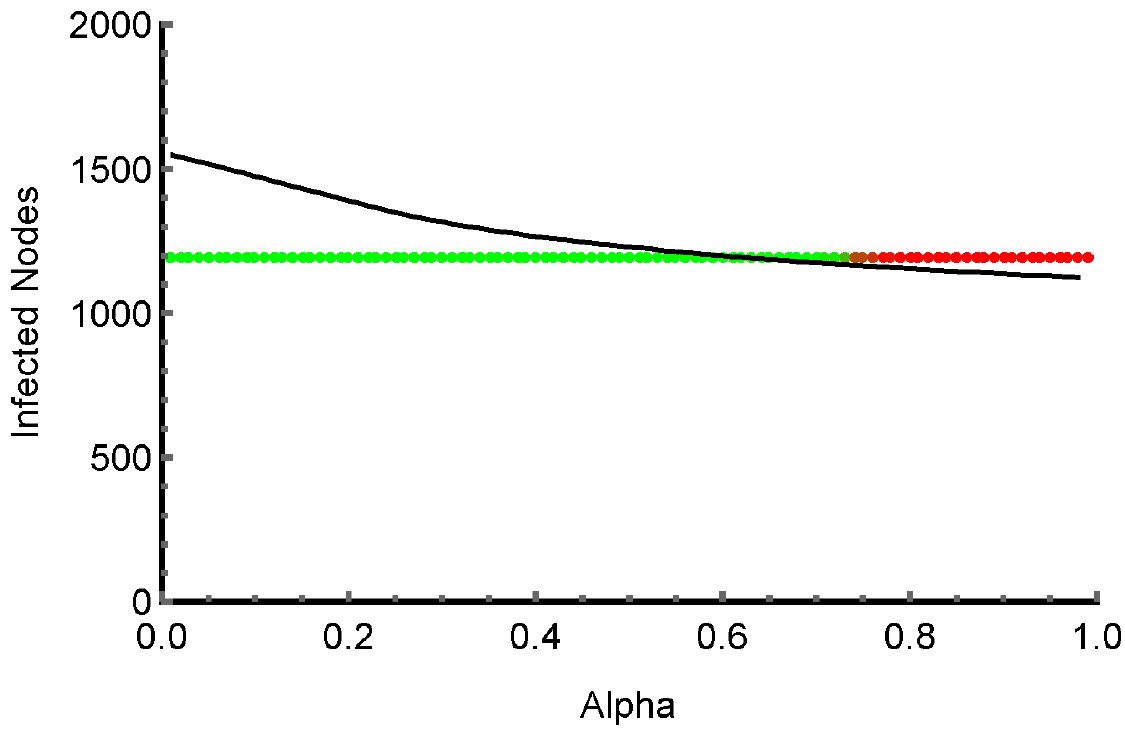}
\subcaption{Bolster Strategy B}
\end{subfigure}
\caption{\label{fig:bolstercomparison} Black line corresponds to hypothetical value of $\I(\tau)$ that would lead to spread of percolation. Red-green line corresponds to the actual value of $\I(\tau)$. Green dots imply percolation stopped and red dots imply the spread was complete. Notice that when the black value is greater than $\I(\tau)$, the percolation stops but when the black value is less than $\I(\tau)$, the percolation spreads.}
\end{figure}

Notice that the \bolstera\ theoretical line is substantially steeper than the \bolsterb\ line. This is the reason the \bolsterb\ intervention is so noisy, and also the reason why \bolstera\ is a stronger intervention than \bolsterb\ ; small changes in $\alpha$ dramatically improve the strength of the intervention.

We can perform the same analysis on \diminish\ and \sequester\ . Recall that \diminish\ deletes every edge with probability $1-\alpha$, whereas \sequester\ only deletes edges connected to an infected vertex. Figure~\ref{fig:dimseqk1} shows these results, using $p=15/n$ and $r=3$ uniformly.

\begin{figure}[H]
\begin{subfigure}[t]{0.5\textwidth}
\centering
\includegraphics[scale=0.45]{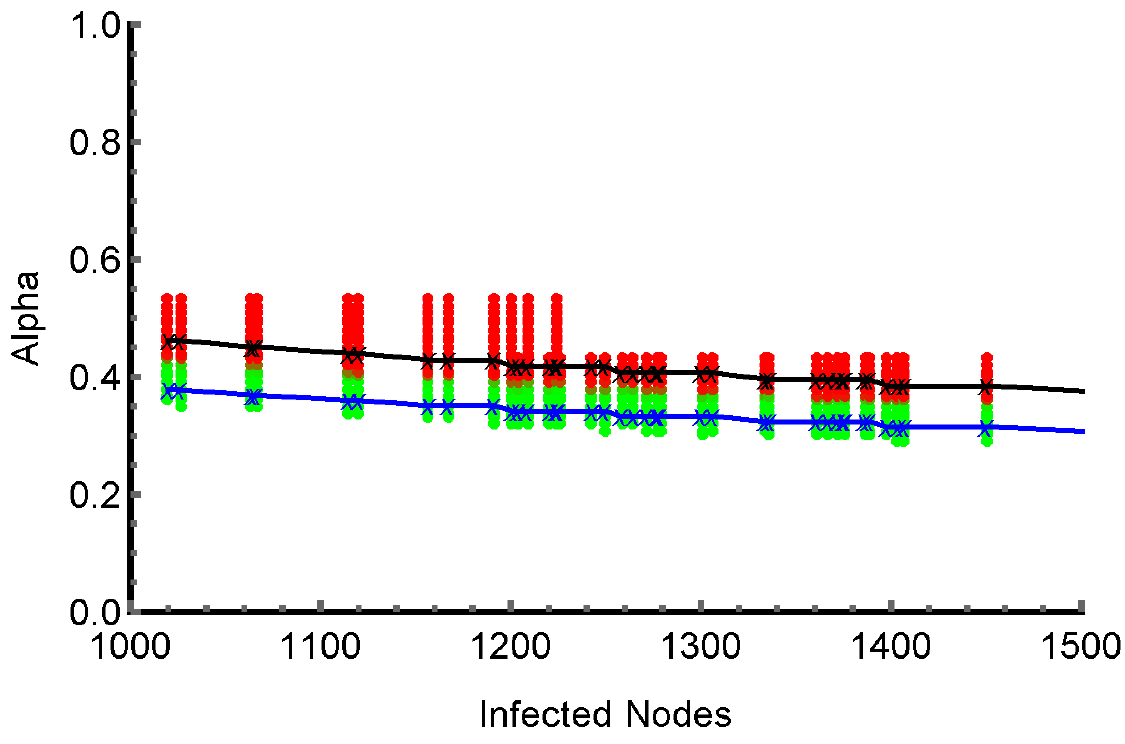}
\subcaption{\label{fig:diminish}\diminish}
\end{subfigure}
\begin{subfigure}[t]{0.5\textwidth}
\centering
\includegraphics[scale=0.45]{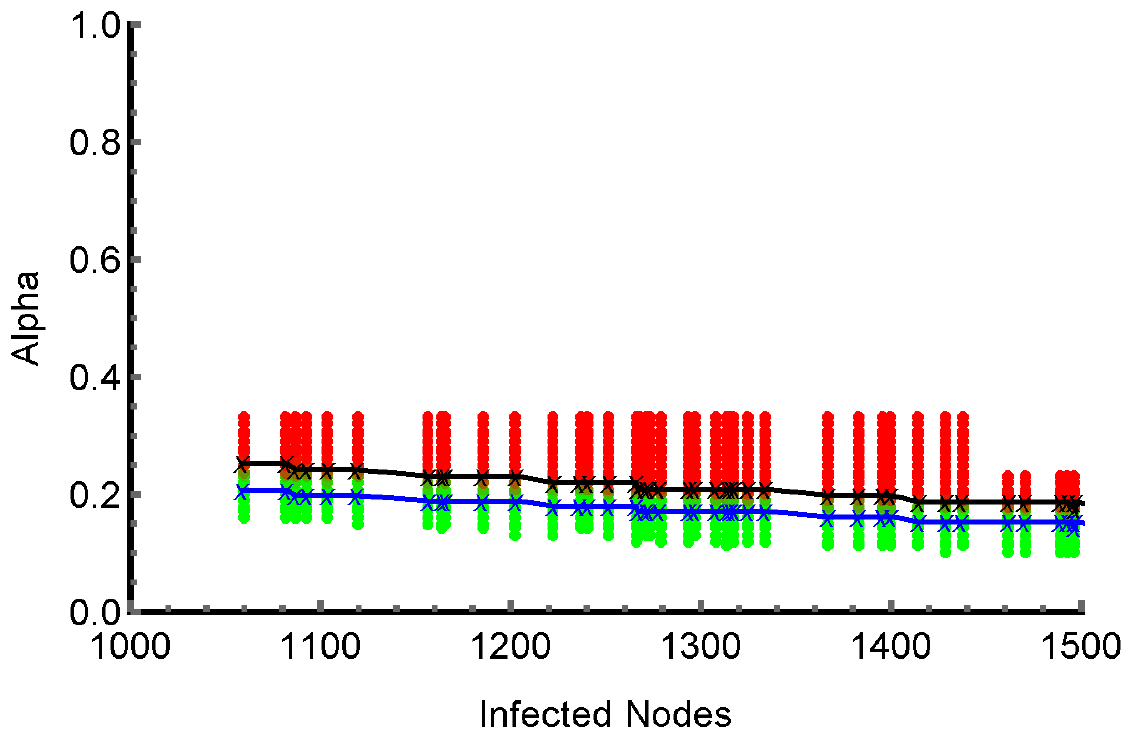}
\subcaption{\label{fig:sequester} \sequester}
\end{subfigure}
\caption{\label{fig:dimseqk1} The $x$-axis and vertical lines have the same meaning as in Figure~\ref{fig:bolsterk1}. Here the edges are dropped with probability $1-\alpha$.}
\end{figure}

Notice that \diminish\ is a substantially stronger intervention than \sequester\, which is expected as \diminish\ deletes all edges whereas \sequester\ only deletes a subset of the edges. The figure for $k=10$ is similar and we omit it for space. Instead, for $k=10$, $p=9/3000$ and $q=6/3000$, we will perform a similar analysis as Figure~\ref{fig:bolstercomparison} and zoom in on a single graph. When $\alpha_p = \alpha_q = \alpha$, we obtain Figure~\ref{fig:dimseqcompare}. 

\begin{figure}[H]
\begin{subfigure}[t]{0.5\textwidth}
\centering
\includegraphics[scale=0.45]{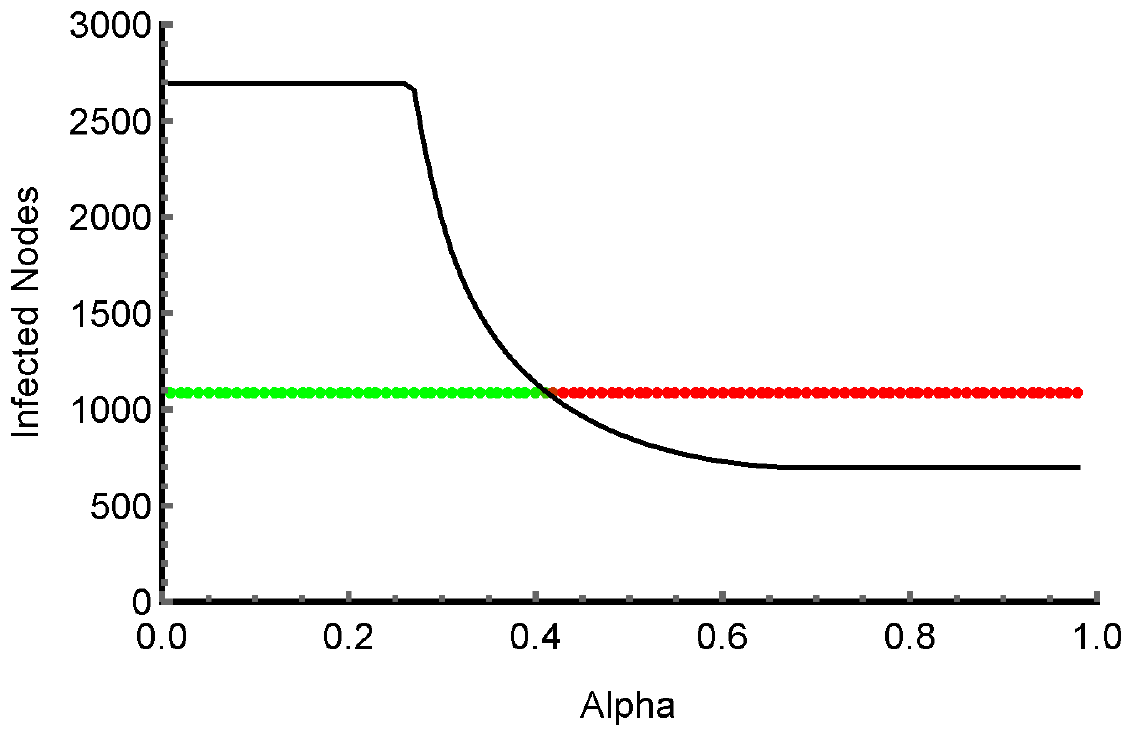}
\subcaption{Diminish}
\end{subfigure}
\begin{subfigure}[t]{0.5\textwidth}
\centering
\includegraphics[scale=0.45]{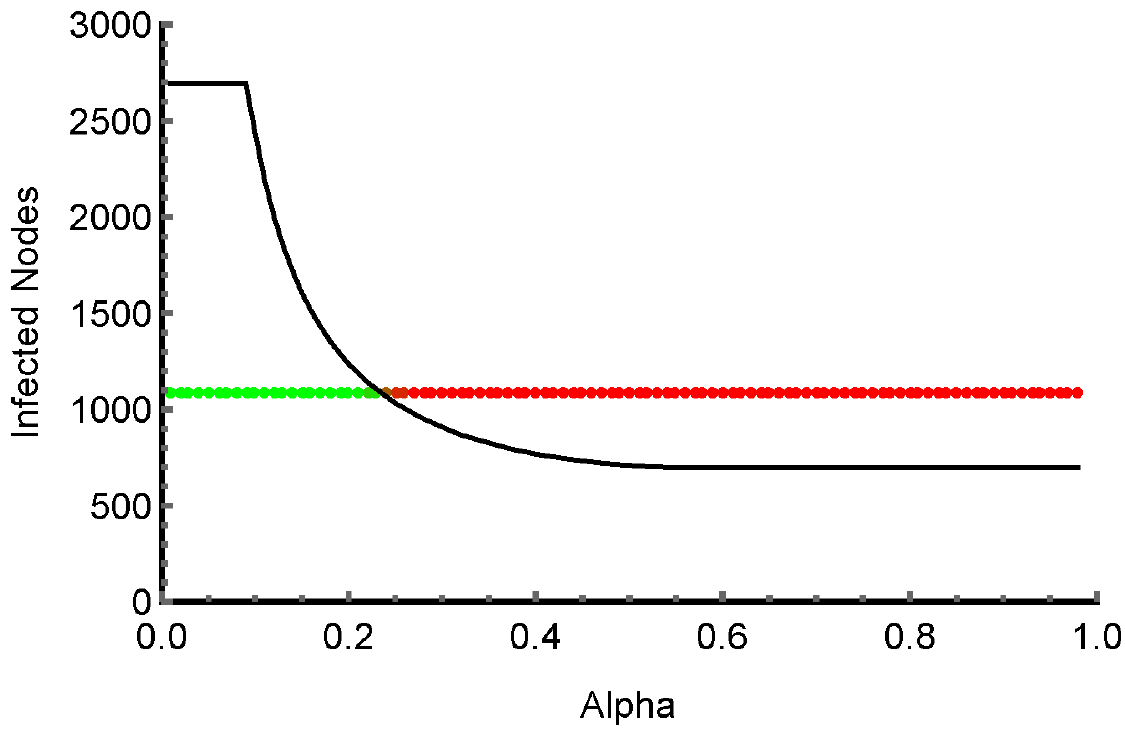}
\subcaption{Sequester}
\end{subfigure}
\caption{\label{fig:dimseqcompare} Edges are dropped with probability $1-\alpha_p$ or $1-\alpha_q$ where $\alpha_p = \alpha_q=\alpha$. The lines have the same meaning as Figure~\ref{fig:bolstercomparison} }
\end{figure}

Our results also hold for the case where $\alpha_p \neq \alpha_q$, which is depicted in Figure~\ref{fig:dimseqcomparediffalpha} (using the same graph as Figure~\ref{fig:dimseqcompare} for ease of comparison).

\begin{figure}[H]
\begin{subfigure}[t]{0.5\textwidth}
\centering
\includegraphics[scale=0.45]{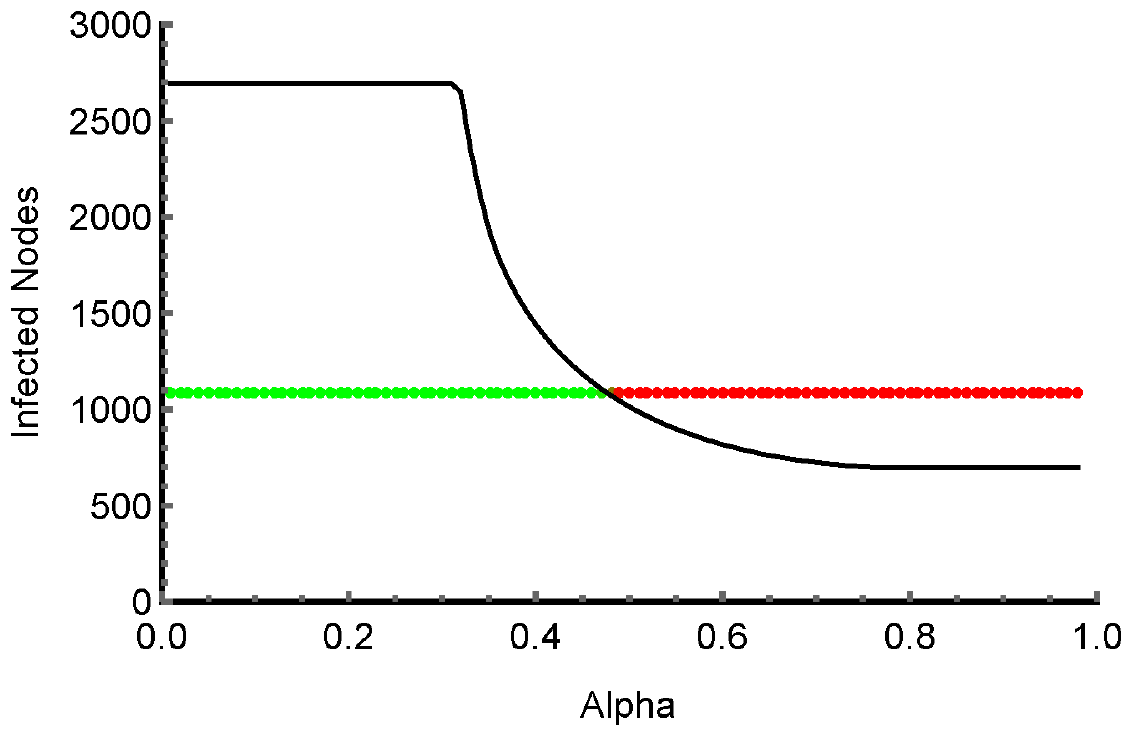}
\subcaption{Diminish}
\end{subfigure}
\begin{subfigure}[t]{0.5\textwidth}
\centering
\includegraphics[scale=0.45]{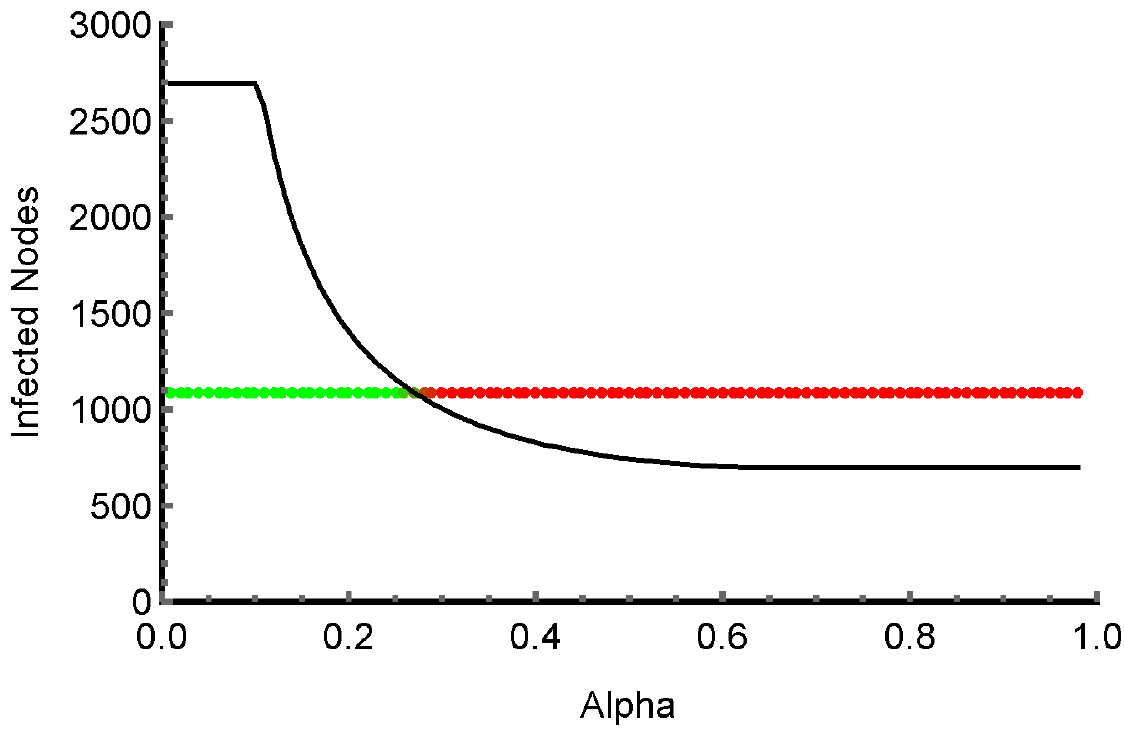}
\subcaption{Sequester}
\end{subfigure}
\caption{\label{fig:dimseqcomparediffalpha} Edges are dropped with probability $1-\alpha_p$ or $1-\alpha_q$ where $\alpha_p = \alpha$ and $\alpha_q = (2/3) \alpha$. The lines have the same meaning as Figure~\ref{fig:bolstercomparison}}
\end{figure}

\paragraph{Summary.} We can construct various intervention strategies and accurately predict whether the percolation will halt or spread. We can also compare various intervention strategies to determine which strategy is more effective.

\section{Proof of Theorem~1}
\label{sec:percolationthresholds}

\label{subsec:easier}
Recall the definition of Templated Multisection graphs in
Definition~\ref{def:tmgraph} in Page \pageref{def:tmgraph}.

\begin{definition}
\label{def:maindef}
Let $\phi=p k_p + q k_p$ and 
$\pi_r(t) =  \Pr[ \bin(k_p t, p) + \bin(k_q t, q) \geq r]$
\begin{eqnarray*}
A(t) = \sum_{r=1}^{\rmax} \zeta_r \pi_r(t)& \qquad \qquad \qquad & f(\seed,t) = (n - \seed)A(t) - kt + \seed \\
\tc(\seed) = \argmin_{t \leq 1/(3\phi)} f(\seed,t)& &\pthreshold = \min_\seed \{ \seed | \forall \ t \leq 1/(3\phi),  f(\seed,t) \geq 0 \} 
\end{eqnarray*}
Note $\tc=\tc(\pthreshold)$. Observe that $\eta\phi$ is the expected degree.
\end{definition}

Theorem~\ref{thm:mainthm} follows from Theorem~\ref{thm:convexity} and
Theorem~\ref{thm:percolationthreshold}. Theorem~\ref{thm:convexity}
proves the existence of $\pthreshold$ and
Theorem~\ref{thm:percolationthreshold} shows that this seed value
shows the desired sharp dichotomy.

\begin{theorem}[Proved in Section~\ref{proof:thm:convexity}]
\label{thm:convexity} 
If $\zeta_1 < 2 \zeta_2/3$, $p,q\leq 1/2$ and $\phi t \leq 1/3$ then
$A(t)$ is convex. If $\seed_ 1 < \seed_2$ then 
$\tc(\seed_1) \leq \tc(\seed_2)$. 
Moreover if for some constant $\beta>0$ we have
$\zeta_1 \eta\phi \leq 1-\beta$ and $\eta \phi = o(\sqrt{\beta n/k})$ then 
$\tc \geq \frac{\beta n}{2k(\phi \eta)^2} \rightarrow \infty$ as $n\rightarrow \infty$.
\end{theorem}

\begin{theorem}
Let $\delta,\epsilon>0$.
Let $G$ be a $TM$ graph with sufficiently large number of nodes 
$n \geq n_0(\delta,\beta,\epsilon,k)$. Suppose we choose $\seed$ vertices uniformly at random and
set them as infected. If $\seed < (1-\epsilon)\pthreshold$ then $G$ 
does not become becomes infected with probability at least $1-O(\epsilon^{-2}/(\tc\beta^{2(\rmax+1)}))$.
If $\seed > (1+\epsilon)\pthreshold$ then an absolute constant fraction of 
the nodes in $G$ become infected with probability at least $1-O(\epsilon^{-2}/(\tc\beta^{2(\rmax+1)}))-O(1/\tc)$.
Moreover if the 
expected degree $\phi \eta$ is a slowly growing function then with 
probability $1-O(\epsilon^{-2}/(\tc\beta^{2(\rmax+1)})) - O(1/\tc) - 1/\eta^{\Omega(1)} = 1-O(\epsilon^{-2}/(\tc\beta^{2(\rmax+1)}))
$, the percolation does not stop till $\eta -o(\eta)$ nodes are infected.
\label{thm:percolationthreshold}
\end{theorem}

\paragraph{Forced Linearizations.} 
We begin our proof of Theroem~\ref{thm:percolationthreshold} by defining two notions: Halting and Cheating three-stage
percolations. Halting percolation is pessimistic: it stops the moment
it encounters a problem. If $G$ is infected by halting percolation, it
will be infected.

\begin{definition}[{\sc Halting Three-Stage Percolation}]
Let $G$ be a Templated Multisection graph where every vertex $u$ has
threshold $r(u)$. Vertices can have three states: healthy, latent, and
contagious. At timestep $0$, select $\seed$ uniformly across the graph
and mark them as latent. Mark all other vertices as healthy. At every
timestep, choose one latent vertex in every cluster and mark it as
contagious. Then, every healthy vertex $u$ with $r(u)$ or more
contagious neighbors become latent. The process terminates the first
time \emph{any} cluster has zero latent vertices.
\end{definition}

Our second definition is Cheating Three-Stage Percolation. Cheating
percolation is optimistic: it cheats by making vertices contagious
even if they have fewer neighbors (than the corresponding thresholds)
infected. If $G$ is not infected by cheating percolation, it will not
be infected.

\begin{definition}[{\sc Cheating Three-Stage Percolation}]
Use the same initialization as Halting Three-Stage Percolation. At
every timestep, choose one latent vertex in every cluster and mark it
as contagious. If there are no latent vertices in a cluster, instead
choose one healthy vertex in that cluster and mark it as
contagious. The process terminates the first time \emph{every} cluster
has zero latent vertices.
\end{definition}

For $k=1$ the two definitions coincide and are the same. 
Theorem~\ref{thm:percolationthreshold} follows from 
Lemma~\ref{lem:upper} and Lemma~\ref{lem:lower}. The next lemma addresses the growth in $\pi_r(t)$.

\begin{lemma}[Proved in Section~\ref{proof:pirlemma}]
\label{lem:pirlemma}
For any $p\geq q, x\geq 1$ and any $t\geq 4r$ with $\phi xt \leq 1/3$, we have $\pi_r( xt) \leq 3\left( \frac{4x}{3(1-p)} \right)^r \pi_r(t)$
and if $3x(1-p)>4$ then
$\pi_r(t) \leq 4\left(\frac{4}{3x(1-p)}\right)^r \pi_r(xt)$.
\end{lemma}

\begin{definition}
\label{def:sdef}
Let $\seed^i_r$ be the number of seeded vertices in cluster $i$ with threshold $r$.
Let $S^i_r(t)$ to be the number of non-seeded vertices in cluster $i$ that have
threshold $r$ and have $r$ or more infected neighbors then $S^i_r(t)$ is a random
variable which is  $\bin(\eta_r - \seed^i_r, \pi_r(t))$ where $\eta_r$ is the number of vertices with threshold $r$. Note $\E[\eta_r] = \zeta_r \eta = \zeta_r n/k$. 
Let $S^i(t) = \sum_{r=2}^{\rmax} S^i_r(t)$ and $S(t)=\sum_i S^i(t)$.
\end{definition}

\noindent The arguments in \cite{gnp} for a fixed threshold can be modified to prove the next lemma, it pretends that the percolation for different 
thresholds are proceeding simultaneously. For a fixed threshold the
derivation uses a martingale argument and Doob's $L_2$ inequality
which bounds the deviation of the entire trajectory from the
expectation. However martingales are preserved under addition -- 
we bound the per step maximum value for which we
use Lemma~\ref{lem:pirlemma} (first part).

\begin{lemma}[Proved in Section~\ref{proof:StSclose}]
Let $t_0=\frac{96}{(1-p)^2\beta^2} \tc$.
For $q\leq p\leq1/2$ and all fixed $\gamma,\beta>0$ if $n \geq n_0(\beta,\gamma,k)$ which is  sufficiently large, with probability at least $1-c_0/(\gamma^2 \beta^{2\rmax} t^*)$ for some absolute constant $c_0$, 
simultaneously for all $i$,
\begin{equation*}
\sup_{1 \leq t \leq t_0} \left | S^i( t) - \left(\eta - \frac{\seed}{k} \right) A(t) \right| \leq \gamma  \tc
\end{equation*}
\label{lem:StSclose}
\end{lemma}

\noindent The next lemma follows from using the second part of
Lemma~\ref{lem:pirlemma}.

\begin{lemma}
\label{lem:constantfactor}
$\pthreshold$ is not too small, i.e., 
$\frac{9(1-p)^2\beta^2}{128}k\tc \leq \pthreshold$. Note $p\leq 1/2$ and $\beta>0$.
\end{lemma}

\begin{proof}
Recall $f(\seed,t)= (n-\seed)A(t) -  kt + \seed$ and $\tc(\seed)$ is the value of $t$ that minimizes $f(\seed,t)$.
Consider decreasing $\pthreshold$ to be a fractional value $\pthreshold'$ such that 
\[ f(\pthreshold',\tc)=(n-\pthreshold')A(\tc) - k\tc + \pthreshold' = 0 \qquad \forall t \leq 1/(3\phi) f(\pthreshold',t) \geq 0 \]
Now $ (n-\pthreshold')A(\tc) = k\tc - \pthreshold'$.
Set $z=\frac{32}{9(1-p)^2\beta}$ and $f(\pthreshold',\tc/z) \geq 0$ rewrites as 
\begin{eqnarray*}
0 & \leq & f(\pthreshold',\tc/z)  = (n-\pthreshold')A(\tc/z) - k\tc/z + \pthreshold' \\
& = & (n- \pthreshold')\zeta_1 \pi_1(\tc/z) + (n - \pthreshold') \sum_{r=2}^{\rmax} \zeta_r \pi_r(\tc/z) - k\tc/z + \pthreshold' \qquad \mbox{(Expanding $A$)}
\end{eqnarray*}
Now for $t\geq 1$, we have $\pi_1(t) = 1 - (1-p)^{k_pt}(1-q)^{k_qt} \leq 1 - (1 - pk_pt)(1-qk_qt) \leq pk_pt + qk_qt = t\phi$ (note $(1-p)^z \geq 1 - pz$ for all $z\geq 1$). Therefore
\begin{eqnarray*}
0 & \leq & n \zeta_1 \phi \tc/z +  (n-\pthreshold') \sum_{r=2}^{\rmax} \zeta_r \pi_r(\tc/z) - k\tc/z + \pthreshold'\\
& \leq &  (n-\pthreshold') \sum_{r=2}^{\rmax} \zeta_r \pi_r(\tc/z) - \beta  k\tc/z + \pthreshold' \qquad \mbox{(Since $\zeta_1 \phi\eta =\zeta_1 \phi n/k \leq (1-\beta)$)}\\
& \leq  &\frac{64}{9z^2(1-p)^2} (n-\pthreshold')\sum_{r=2}^{\rmax} \zeta_r \eta \pi_r(\tc) - \beta k\tc/z + \pthreshold'  \qquad \mbox{(Using Lemma~\ref{lem:pirlemma} with $x=1/z$.)}\\
& =  & \frac{64}{9z^2(1-p)^2} (n-\pthreshold')A(\tc) - \beta k\tc/z + \pthreshold' \\
& \leq & \beta k\tc/(2z) - \beta k\tc/z + \pthreshold' = - \beta k\tc/(2z) + \pthreshold'
\end{eqnarray*}
\end{proof}

\begin{lemma}
\label{lem:lower}
If $\seed < (1-\epsilon)\pthreshold$ then the cheating percolation stops with probability 
$1-O(\epsilon^{-2}/(\tc\beta^{2(\rmax+1)}))$ for all sufficiently large $n$.
\end{lemma}
\begin{proof}
Consider decreasing $\pthreshold$ to be a fractional value $\pthreshold'$ 
such that $\forall t \leq 1/(2p) f(\pthreshold',\tc) \geq 0$ and 
\[ f(\pthreshold',\tc)=(n-\pthreshold')A(\tc) - k\tc + \pthreshold' = 0 \]
Note $\pthreshold' \geq \pthreshold - 1$. 
(This step is also helpful in proving Lemma~\ref{lem:constantfactor}). 
Observe that $\pthreshold' \leq k\tc$ since $A(t)$ is non-negative. Therefore $n A(\tc) \leq k\tc + \pthreshold' A(\tc) \leq 2 k\tc$.
But notice that we assumed $\eta\phi \geq 4$ and $t\phi \leq 1/3$ and therefore $k\tc \leq n/12$. Therefore $A(\tc) \leq 1/4$.

Let $\underline{\seed}= \pthreshold' - \epsilon \pthreshold$ (differs from $(1-\epsilon)\pthreshold$ by at most $1$)
and $\underline{t}=\tc(\underline{\seed})$. From Theorem~\ref{thm:convexity} $\underline{t} \leq \tc$.
\begin{eqnarray*}
(n-\pthreshold')A(\tc) -  k\tc + \pthreshold' =  0   \qquad 
& \implies&  
(n-\underline{\seed})A(\tc) -  k\tc + \underline{\seed}  = - (\pthreshold' - \underline{\seed})(1 - A(\tc))  \\
&\implies &
(n-\underline{\seed})A(\underline{t}) -  k\underline{t} + \underline{\seed} \leq  - (\pthreshold' - \underline{\seed})(1 - A(\tc))
\end{eqnarray*}
The last line follows from the fact that in the range $[\underline{t},\tc]$ the function 
$f(\underline{\seed},t)$ is increasing. But since $A(\tc) \leq 1/4$ we now have that
\[ 
(n-\underline{\seed})A(\underline{t}) -  k\underline{t} + \underline{\seed} \leq - 3 \epsilon \pthreshold/4 
\]
Using $\gamma=\frac{\epsilon 9(1-p)^2\beta^2}{128}$ in Lemma~\ref{lem:StSclose} 
with probability $1-O(\epsilon^2/\tc\beta^{2(\rmax+1)})$ for every cluster $i$,
\[ \sup_{1 \leq t \leq t_0} \left | S^i( t) - \E[S^i(t)] \right| \leq \frac{\epsilon 9(1-p)^2\beta^2}{128} \tc \leq \frac{\epsilon}{2k} \pthreshold  \qquad \mbox{(Using Lemma~\ref{lem:constantfactor}.)}\]

\noindent Therefore $S^i(\underline{t}) \leq \E[S^i(\underline{t})] + \frac{\epsilon}{2k} \pthreshold$. Let $\underline{\seed^i}$ be the number of seed vertices in cluster $i$. 
Using Chernoff bounds we can assert that with probability $(1-\delta/(2k))$ 
$\underline{\seed^i} \leq (1+\epsilon/8)\underline{\seed}/k$ -- observe that this 
result will hold when $\underline{\seed} \geq \frac{3k}{\epsilon^2} \ln \frac{2k}{\delta}$ but $\underline{\seed} = \Omega(t^*) \rightarrow \infty$. 
Therefore using union bound with probability at least $1-\delta$, for every cluster $i$,

\begin{eqnarray*}
S^i(\underline{t}) + \underline{\seed^i} & \leq & \E[S^i(\underline{t})] + (1+\frac{\epsilon}{8}) \frac1k \underline{\seed} + \frac{\epsilon}{2k} \pthreshold =  \frac1k(n-\underline{\seed})A(\underline{t})  + \frac1k\underline{\seed} + \frac1k\frac{5\epsilon}{8} \pthreshold \\
& = & \frac1k\left( (n-\underline{\seed})A(\underline{t})  + \underline{\seed} - k\underline{t} + \frac{5\epsilon}{8} \pthreshold \right) + 
 \underline{t} \leq  \frac1k\left(  - \frac{3 \epsilon}{4}  \pthreshold + \frac{5\epsilon}{8} \pthreshold \right) + \underline{t} <  \underline{t}
\end{eqnarray*}
Therefore with probability at least $1-\delta$ the percolation stops {\em in every cluster} before $\underline{t}$. 
\end{proof}

\paragraph{The large seed case:} In the other case we show that if the
percolation survives sufficiently past the bottleneck region then it
leads to complete percolation.  Note that for the following lemma we
can assume that we started with a seed $\overline{\seed} =
(1+\epsilon)\pthreshold$ and $\epsilon$ is small. If the seed size is
larger we can simply ignore the remaining nodes.
The proof is broken into three lemmas, culminating in Lemma~\ref{lem:upper}.

\begin{lemma}[Proved in Section~\ref{proof:pb1}]
\label{lem:pastbottleneck1}
When the seed size is $\overline{\seed}=(1+\epsilon)\pthreshold$, $\epsilon\leq 1/9$ and we have reached $t=\min\{\frac{96}{(1-p)^2\beta^2}\tc,1/(3\phi)\}$
then with
probability $1-O(\frac1{\tc})$, $S^i(t) > t-a$ for all $t \in [\frac{96}{(1-p)^2\beta^2}\tc, 1/(3\phi)]$, i.e., as $n$ increases and $\tc \rightarrow \infty$ the percolation continues till $t=1/(3\phi)$.
\end{lemma}

\begin{lemma}[Proved in Section~\ref{proof:closer}]
\label{lem:closer}
If the percolation has continued till $t=(3\phi)^{-1}$ then with
probability $1-1/n^{\Omega(1)}$, the percolation does not stop till a constant fraction of the graph is infected. Moreover if the expected degree $\phi \eta$ is a slowly growing function then with probability $1-1/n^{\Omega(1)}$, the percolation does not stop till $\eta -o(\eta)$ nodes are infected.
\end{lemma}

\begin{lemma}
\label{lem:upper}
If the expected degree is at least $2$ and $\seed > (1+\epsilon)\pthreshold$ then for sufficiently large $n$ with probability $1-O(\epsilon^{-2}/\beta^{2(\rmax+1)} t^*)$
the halting percolation continues till 
an absolute constant fraction of the nodes are infected. Moreover if the 
expected degree $\phi \eta$ is a slowly growing function then with probability 
$1-O(\epsilon^{-2}/\beta^{2(\rmax+1)} t^*)$, the percolation does not stop till $\eta -o(\eta)$ nodes are infected.
\end{lemma}

\begin{proof}
As in the proof of Lemma~\ref{lem:lower}, 
let $\overline{\seed}= \pthreshold' + \epsilon \tc$ and $\overline{t}=\tc(\overline{\seed})$. By Theorem~\ref{thm:convexity}, $\tc \leq \overline{t}$. 
Suppose that $
 (n-\overline{\seed})A(\overline{t}) -  k\overline{t} + \overline{\seed} \leq  3 \epsilon \pthreshold/4 
 $ then since  $\overline{t},\pthreshold/k$ are at most $1/(3\phi) \leq nk/6 $ when $\epsilon \leq 1$ we again have
 $A(\overline{t}) \leq 1/4$.
Suppose not. Then assume for contradiction,

\[ (n-\overline{\seed})A(\overline{t}) -  \overline{t} + \overline{\seed} < \frac{3 \epsilon}{4} \pthreshold \leq \frac{3\epsilon}{4} {\tc} \leq \frac{3\epsilon}{4}\frac{1}{3\phi} \leq \frac{3\epsilon}{4}\frac{\eta}{3\phi\eta} \leq \frac{3\epsilon}{4}\frac{n}{6k}
\]
(assuming that the expected degree is at least $2$)
which implies (since $A(t) \leq 1$ for all $t$ and $\overline{\seed}(1-A(t)) \geq 0$)
\[ nA(\overline{t}) \leq \frac{3\epsilon}{4}\frac{n}{6k} + \overline{t} \leq \frac{3\epsilon}{4}\frac{n}{6k} + \frac{n}{6k}
\leq \frac{3\epsilon}{4}\frac{n}{6k} + \frac{n}{6k} \leq \frac{n}{4} 
\]
which implies that $A(\overline{t}) \leq 1/4$ when $\epsilon \leq 1$. 
Now using definition of $\pthreshold'$, at $t=\overline{t}$,

\begin{eqnarray*}
(n-\pthreshold')A(\overline{t}) -  \overline{t} + \pthreshold' \geq  0 
& \implies &  
(n-\overline{\seed})A(\overline{t}) -  \overline{t} + \overline{\seed} \geq  (\overline{\seed} - \pthreshold')(1 - A(\overline{t}))  \\
& \implies & 
(n-\overline{\seed})A(\overline{t}) -  \overline{t} + \overline{\seed} \geq  \frac{3\epsilon}{4} \pthreshold
\end{eqnarray*}
which is a contradiction. Since $\overline{t}$ was the minimum,
\begin{equation}
\label{eqn:whew}
(n-\overline{\seed})A(t) -  t + \overline{\seed} \geq  \frac{3 \epsilon}{4} \pthreshold \qquad \forall t\leq 1/(3\phi)
\end{equation}
Again as in the proof of Lemma~\ref{lem:lower}, 
with probability $1-O(\epsilon^2/\tc\beta^{2(\rmax+1)})$ for every cluster $i$,
\[ \sup_{1 \leq t \leq t_0} \left | S^i( t) - \E[S^i(t)] \right| \leq \frac{\epsilon 9(1-p)^2\beta^2}{128} \tc \leq \frac{\epsilon}{2k} \pthreshold \]
Therefore with probability  $1-O(\epsilon^2/\tc\beta^{2(\rmax+1)})$ for all $i$ and $t\leq t_0$
\begin{eqnarray*}
 S^i(t) + \overline{\seed}^i & \geq & \frac1k\left( (n-\overline{\seed})A(t)  + \overline{\seed} \right) - \frac{\epsilon}{2k} \pthreshold \\
& \geq & \frac1k\left( \frac{3 \epsilon}{4} \pthreshold - \frac{ \epsilon}{2} \pthreshold \right) + t  \geq t \qquad \mbox{(Using Equation~\ref{eqn:whew}.)}
\end{eqnarray*}
which implies that the percolation does not stop till $t_0$. To complete the proof we now use Lemmas~\ref{lem:pastbottleneck1} and \ref{lem:closer}. 
\end{proof}

\subsection{Omitted Proofs}
\subsubsection{Proof of Theorem~\ref{thm:convexity}}
\label{proof:thm:convexity}

\begin{ndefinition}{\ref{def:maindef}} 
Let $\phi=p k_p + q k_p$ and 
$\pi_r(t) =  \Pr[ \bin(k_p t, p) + \bin(k_q t, q) \geq r]$
\begin{eqnarray*}
A(t) = \sum_{r=1}^{\rmax} \zeta_r \pi_r(t)& \qquad \qquad \qquad & f(\seed,t) = (n - \seed)A(t) - kt + \seed \\
\tc(\seed) = \argmin_{t \leq 1/(3\phi)} f(\seed,t)& &\pthreshold = \min_\seed \{ \seed | \forall \ t \leq 1/(3\phi),  f(\seed,t) \geq 0 \} 
\end{eqnarray*}
Note $\tc=\tc(\pthreshold)$. Observe that $\eta\phi$ is the expected degree.
\end{ndefinition}

\medskip
\begin{ntheorem}{\ref{thm:convexity}}
If $\zeta_1 < 2 \zeta_2/3$, $p,q\leq 1/2$ and $\phi t \leq 1/3$ then
$A(t)$ is convex.  If $\seed_ 1 < \seed_2$ then 
$\tc(\seed_1) \leq \tc(\seed_2)$. Moreover if for some constant $\beta>0$ we have
$\zeta_1 \eta\phi \leq 1-\beta$ and $\eta \phi = o(\sqrt{\beta n/k})$ then 
$\tc \geq \frac{\beta n}{2k(\phi \eta)^2} \rightarrow \infty$ as $n\rightarrow \infty$.
\end{ntheorem}

\noindent Theorem~\ref{thm:convexity} follows from the next two theorems.
We state both theorems and prove the latter theorem
(Theorem~\ref{thm:increase}) first since the former
(Theorem~\ref{thm:basic}) is a detailed verification of properties of
binomial coefficients and Theorem~\ref{thm:increase} relies on
Theorem~\ref{thm:basic}.

\begin{theorem}
\label{thm:basic}
When $\zeta_1 \leq 2\zeta_2/3$ $p,q\leq 1/2$ and $\phi t \leq 1/3$ then $A(t)$ is convex.
\end{theorem}

\begin{theorem}
\label{thm:increase}
If $\seed_ 1 < \seed_2$ then 
$\tc(\seed_1) \leq \tc(\seed_2)$. 
Further when $\zeta_1 \leq 2\zeta_2/3$, and for some constant $\beta>0$ (i) the expected degree $\eta \phi$ satisfies $\eta\phi \leq \sqrt{\beta n/k}$ and (ii) the fraction $\zeta_1$ of vertices with threshold $1$ satisfies $\zeta_1 \eta \phi \leq 1 - \beta$ 
then as $n \rightarrow \infty$ implies $\tc \geq \frac{\beta n}{2k(\phi \eta)^2} \rightarrow \infty$.
\end{theorem}
\begin{proof}({\bf Of Theorem~\ref{thm:increase}.})
We use Theorem~\ref{thm:basic} to prove $A(t)$ to be a convex
function. Note $A(0)=0$. Suppose that we could find another convex
function $\bar{A}(t)$ such that $\bar{A}(0)=0$ and for all $t\leq
1/(3\phi), \bar{A}(t) \geq A(t)$.

Define $\bar{f}(\varphi,t)=(n-\varphi)\bar{A}(t) -kt + \varphi$ and
let $\pthreshold' = \min_\seed \{ \seed | \forall \ t \leq 1/(3\phi),
\bar{f}(\seed,t) \geq 0 \}$. Note that $\pthreshold' \leq \pthreshold$
since for all $\seed$ we have $\bar{f}(\seed,t) \geq f(\seed, t)$. 

Now for a fixed seed $\varphi$, the value $\tc(\varphi)$ (extended to
the reals from integers) corresponds to the point where the slope of
$(n-\varphi)A(t)$ is $k$. Consider simultaneously the functions
$(n-\pthreshold')\bar{A}(t), (n-\pthreshold')A(t)$ and
$(n-\pthreshold)A(t)$ and the corresponding points where they are
tangent to a line with slope $k$. This is shown if
Figure~\ref{fig:bound}.

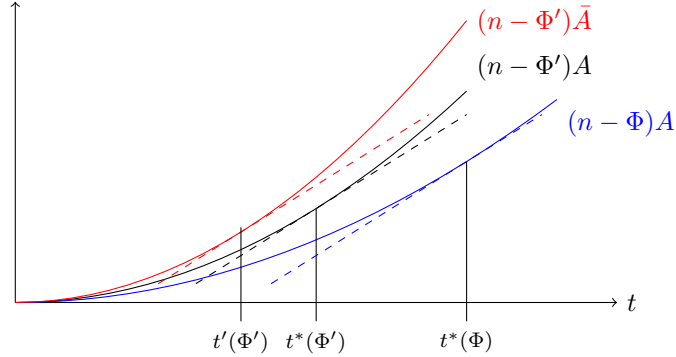
\begin{figure}[H]
\begin{center}
 \begin{tikzpicture}[xscale=0.4,yscale=0.5]
  \draw[->] (0,0) -- (20,0) node[right] {\small$t$};
  \draw[->] (0,0) -- (0,8) node[above] {\small$$};

  \draw[domain=0:15] plot (\x,{(\x*\x/40)}) 
    node[above right] {\small$(n-\pthreshold')A$};
  \draw[blue,domain=0:18] plot (\x,{(\x*\x/60)}) 
    node[below right] {\small$(n-\pthreshold)A$};
  \draw[red,domain=0:15] plot (\x,{(\x*\x/30)}) 
    node[right] {\small$(n-\pthreshold')\bar{A}$};
  \draw[dashed] (6,0.5) -- (15,5);
  \draw[-] (10,2.5) -- (10,-0.5) node[below] {\scriptsize $\tc(\pthreshold')$};
  \draw[red,dashed] (4.75,0.5) -- (13.75,5);
  \draw[-] (7.5,2) -- (7.5,-0.5) node[below] {\scriptsize $t'(\pthreshold')$};
  \draw[blue,dashed] (8.5,0.5) -- (17.5,5);
  \draw[-] (15,3.75) -- (15,-0.5) node[below] {\scriptsize $\tc(\pthreshold)$};
 \end{tikzpicture}
\end{center}
\caption{\label{fig:bound}Bounding $\tc$}
\end{figure}

Then $\tc(\pthreshold')$ is bounded below by $t'(\pthreshold')$ and 
$\tc(\pthreshold')$ is bounded above by $\tc(\pthreshold)$. Therefore if we prove that $t'(\pthreshold') \rightarrow \infty$ then $\tc=\tc(\pthreshold) \rightarrow \infty$ as well.
Observe that this observation also proves if $\seed_ 1 < \seed_2$ then 
$\tc(\seed_1) \leq \tc(\seed_2)$. 

\medskip
To choose $\bar{A}$ we first observe that if we increase $\zeta_2$ by $\Delta$ and decrease any $\zeta_r$ by $\Delta$ for any $r\geq 3$ then $A(t)$ does 
not decrease and continues to remain convex (note that we continue to satisfy the constraint involving $\zeta_1$). This implies that we can 
assume $\zeta_2=1-\zeta_1$ and let this new function be $A_1(t)$ and by 
construction $A_1(t) \geq A(t)$. Now
\begin{eqnarray*}
A(t) & \leq & A_1(t)  =  \zeta_1 \pi_1(t) + (1-\zeta_1) \pi_2(t) \\
&=& 1 - \Pr[\bin(k_p t, p) + \bin(k_q t, q)=0] - (1-\zeta_1) \Pr[\bin(k_p t, p) + \bin(k_q t, q)=1] \\
& \leq&  1 - e^{-\phi t} - (1-\zeta_1) \phi t e^{-\phi t} + 2p^2 k_p t + 2q^2 k_q t \qquad \mbox{(Le Cam's Theorem \cite{lecam})}
\end{eqnarray*}
where the last step follows from Le Cam's Theorem of approximating the sum of bernoulli distributions by a Poisson process. In this case we were summing $kt$ 
Bernoulli processes of which $k_p t$ had probability $p$ and $k_qt$ had probability $q$. Now $2p^2 k_p  + 2q^2 k_q \leq 2 \phi^2$ and therefore we set:

\[ \bar{A}(t) =  1 - e^{-\phi t} - (1-\zeta_1) \phi t e^{-\phi t} + 2\phi^2 t \]
It is immediate that $A(t) \leq \bar{A}(t)$ and 
\begin{eqnarray*}
\bar{A}'(t) & = &\zeta_1 \phi e^{-\phi t} + (1-\zeta_1) \phi^2 t e^{- \phi t} + 2\phi^2 \\
\bar{A}''(t) & = & \phi^2 e^{-\phi t} \left[ - \zeta_1 + (1-\zeta_1) - \phi t (1-\zeta_1) \right]
\end{eqnarray*}
Based on $\phi t \leq 1/3$ and $\zeta_1 \leq 2(1-\zeta_1)/3$ we have $\bar{A}''(t)\geq 0$ and $\bar{A}$ is convex. At $t=t'(\pthreshold')$
\[ k = (n-\pthreshold') \bar{A}'(t'(\pthreshold')) \leq (n-\pthreshold') \left[ \zeta_1 \phi +  \phi^2 (t'(\pthreshold')+2) \right] \]
Using $\zeta_1 \eta \phi \leq (1-\beta)$ and $\eta=n/k$ we get
\[ t'(\pthreshold') + 2 \geq \frac1{\phi^2} \left[ \frac{k}{n - \pthreshold'} - \frac{(1-\beta)k}{n} \right] = \frac1{\phi^2} \frac{\beta nk + (1-\beta) \pthreshold' k}{(n-\pthreshold')n} \geq \frac1{\phi^2} \frac{\beta nk}{n^2} = \frac{\beta n}{k(\eta\phi)^2}\]
The theorem follows.
\end{proof}

\begin{proof}({\bf Of Theorem~\ref{thm:basic}.}) We begin with some notation.

\begin{definition}
Define $B_i(t)=\Pr[ \bin(k_pt,p) = i]$, $C_j(t)=\Pr
[ \bin(k_qt,q) = j]$ and  $ D_r(t)=\Pr[ \bin(k_pt,p) + \bin(k_qt,q) = r]$. Observe that
\[ D_r(t) = \sum_{i=0}^r B_i(t)C_{r-i}(t) \]
Note $\pi_r(t) = \sum_{r' \geq r} D_{r'}(t)$.
\end{definition}

Let $c_1=\ln(1-p)$. Note $c_1<0$ and when $p\leq1/2$  
$c_1= -p - \frac{p^2}{2} - \frac{p^3}{3} \cdots \geq -3p/2$. Now $B_i(t)={ k_p t \choose i} p^i (1-p)^{k_p t-i}$ and 

\begin{eqnarray*}
B'_i(t) & =& \frac{dB_i(t)}{dt} = { k_p t \choose i} p^i (1-p)^{k_p t-i}   \left[ c_1 + \sum_{j=0}^{i-1} \frac1{k_p t-j} \right] k_p\\
B''_i(t) & = & 
{ k_p t \choose i} p^i (1-p)^{k_pt-i}  \left\{ \left( c_1 + \sum_{j=0}^{i-1} \frac1{k_p t-j} \right)^2 - \sum_{j=0}^{i-1} \frac1{(k_p t-j)^2} \right \} k^2_p \\
 & = & 
 { k_p t \choose i} p^i (1-p)^{k_p t-i} \left\{ c_1^2 + 2c_1 \left( \sum_{j=0}^{i-1} \frac1{k_pt-j} \right) + \left( \sum_{j=0}^{i-1} \frac1{k_p t-j} \right)^2 - \sum_{j=0}^{i-1} \frac1{(k_p t-j)^2} \right \} k_p^2
 \end{eqnarray*}
The above implies that $B'_i(t) \geq 0$ for $i\geq 1$ and $k_p pt \leq \frac13$. Further for $i \geq 2$
 \begin{eqnarray*}
 B''_i(t)
 & \geq & 
 { k_p t \choose i} p^i (1-p)^{k_pt-i}  \left\{ c_1^2 + \frac{2c_1i}{k_pt-i+1} + \left( \sum_{j=1}^{i-1} \sum_{j'=0}^{j-1} \frac2{(k_pt-j)(k_pt-j')} \right) \right\}  k_p^2 \\
 & \geq & 
 { k_pt \choose i} p^i (1-p)^{k_pt-i}  \left\{ c_1^2 + \frac{2c_1i}{k_pt-i+1} + \frac{i(i-1)}{(k_pt-1)k_pt} \right\} k_p^2
 \end{eqnarray*}
 which is positive for $i \geq 2$ and $k_p c_1 t \geq -1/2$. By the exact same 
argument $C_j(t) \geq 0$ for $j \geq 1$ and $C_j''(t) \geq 0$ for all $j \geq 2$.
Now consider $B_i(t)C_0(t)$ for $i \geq 2$, let $c_2 = \ln (1-q)$.

\begin{eqnarray*}
\frac{d B_i(t)C_0(t)}{dt} & = &{ k_p t \choose i} p^i (1-p)^{k_p t-i}(1-q)^{k_q t}   \left[ k_p c_1 + \sum_{j=0}^{i-1} \frac{k_p}{k_p t-j} + k_q c_2 \right] \\
\frac{d^2 B_i(t)C_0(t)}{dt^2} & = & 
{ k_p t \choose i} p^i (1-p)^{k_pt-i} (1-q)^{k_q t} \left( \left[ k_p c_1 + \sum_{j=0}^{i-1} \frac{k_p}{k_p t-j} + k_q c_2 \right]^2 - \sum_{j=0}^{i-1} \frac{k^2_p}{(k_p t-j)^2} \right )
 \end{eqnarray*}
To prove that $\frac{d^2 B_i(t)C_0(t)}{dt^2} \geq 0$ for $i \geq 2$ it suffices to show that
\begin{eqnarray*}
\left[ k_p c_1 + \sum_{j=0}^{i-1} \frac{k_p}{k_p t-j} + k_q c_2 \right]^2 - \sum_{j=0}^{i-1} \frac{k^2_p}{(k_p t-j)^2} \geq 0
\end{eqnarray*}
The left hand side expands to
\begin{eqnarray}
\label{later001}
(k_p c_1 + k_q c_2)^2 + 2(k_p c_1 + k_q c_2)\sum_{j=0}^{i-1} \frac{k_p}{k_p t-j} 
+ \sum_{j=0}^{i-1} \frac{k_p}{k_p t-j} \left(\sum_{j'=0,j'\neq j}^{i-1} \frac{k_p}{k_p t-j'} \right)
\end{eqnarray}
The first term is positive and for $i\geq 2$ for any $j$ there exists a $j \neq j'$ and 
\begin{eqnarray}
\label{later002}
2(k_p c_1 + k_q c_2) + \frac{k_p}{k_p t-j'} \geq 2 \left( - \frac{3p}{2}k_p -  \frac{3q}{2}k_q \right) + \frac1t = \frac{1-3\phi t}{t} \geq 0
\end{eqnarray}
and therefore for $i\geq 2$ we have $\frac{d^2 B_i(t)C_0(t)}{dt^2}
\geq 0$. Note that the same argument holds for $B_0(t)C_j(t)$ for $j
\geq 2$ and $\frac{d^2 B_0(t)C_j(t)}{dt^2} \geq 0$ in that case as
well. Finally consider $B_i(t)C_1(t)$ for $i \geq 1$. 

\begin{eqnarray*}
B_i(t)C_1(t) & = & q k_q t { k_p t \choose i} p^i (1-p)^{k_p t-i}(1-q)^{k_q t-1}\\
\frac{d B_i(t)C_1(t)}{dt} & = & q k_q t { k_p t \choose i} p^i (1-p)^{k_p t-i}(1-q)^{k_q t-1}   \left[ k_p c_1 + \sum_{j=0}^{i-1} \frac{k_p}{k_p t-j} + k_q c_2 + \frac1t \right] \\
\frac{d^2 B_i(t)C_1(t)}{dt^2} & = & B_i(t)C_1(t)
 \left( \left[ k_p c_1 + \sum_{j=0}^{i-1} \frac{k_p}{k_p t-j} + k_q c_2 + \frac1t \right]^2 - \sum_{j=0}^{i-1} \frac{k^2_p}{(k_p t-j)^2} - \frac1{t^2}\right )
 \end{eqnarray*}
Using an expansion similar to Equation~\ref{later001} we wish to prove
\begin{eqnarray*}
\left(k_p c_1 + k_q c_2 + \frac1t \right)^2 + 2 \left(k_p c_1 + k_q c_2 + \frac1t \right)\sum_{j=0}^{i-1} \frac{k_p}{k_p t-j} 
+ \sum_{j=0}^{i-1} \frac{k_p}{k_p t-j} \left(\sum_{j'=0,j'\neq j}^{i-1} \frac{k_p}{k_p t-j'} \right) - \frac1{t^2}\geq 0
\end{eqnarray*}
But the left hand side is greater than
\begin{eqnarray*}
\left(k_p c_1 + k_q c_2 + \frac1t \right)^2 + 2 \left(k_p c_1 + k_q c_2 + \frac1t \right)\frac1t - \frac1{t^2}
\end{eqnarray*}
which rewrites to
\[ (k_p c_1 + k_q c_2)^2 + \frac2{t} \left( \frac1t + 2(k_p c_1 + k_q c_2) \right) \geq (k_p c_1 + k_q c_2)^2 + \frac2{t} \left( \frac1t - 3\phi \right) \geq 0 \]
But the term $k_q q \geq 0$ whereas $c_1, c_2 < 0$ and therefore by the exact same logic as in Equation~\ref{later002},
\begin{eqnarray*}
2(k_p c_1 + k_q c_2 + k_q q) + \frac{k_p}{k_p t-j'} \geq 2 \left( - \frac{3p}{2}k_p -  \frac{3q}{2}k_q \right) + \frac1t = \frac{1-3\phi t}{t} \geq 0
\end{eqnarray*}
Therefore $\frac{d^2 B_i(t)C_1(t)}{dt^2} \geq 0$ for $i\geq 1$. Likewise $\frac{d^2 B_1(t)C_j(t)}{dt^2} \geq 0$ for $j\geq 1$.

\medskip 
\noindent Now for any $D_{r'}(t)$ with $r'\geq 2$ observe that
\[ D_{r'}(t) = B_{r'}(t)C_{0}(t) + B_{r'-1}(t)C_1(t) + \left( \sum_{j=2}^{r'-2}  B_{r'-j}(t)C_{j}(t) \right) + B_{1}(t)C_{r'-1}(t) + B_0(t)C_{r'}(t)\]
The middle sum is a sum of product of convex functions and every other term is proven to be convex as above. Therefore $\pi_r(t) = \sum_{r'\geq r} D_{r'}(t)$ is convex in $t$ for $r\geq 2$.

\medskip
Therefore to show that $A(t)=\sum_{r=1}^{\rmax} \zeta_1 \pi_r(t)$ is convex, we can ignore the terms corresponding to $r\geq 3$. The term corresponding to $r=1$ is not convex however -- but we will show that
\[
\underline{A}(t)=\zeta_1 \pi_1(t) + \zeta_2 \pi_2(t)\]
is convex, which will complete the proof that $A(t)$ is convex. Set $a=\zeta_1/(\zeta_1+\zeta_2)$ and note

\begin{eqnarray*}
\frac{\underline{A}(t)}{\zeta_1 + \zeta_2}  & = & 1 - (1-p)^{k_pt}(1-q)^{k_qt} - (1-a)\left[ pk_pt(1-p)^{k_pt-1}(1-q)^{k_qt} +  (1-p)^{k_p t}qk_qt(1-q)^{k_qt-1}\right] \\
& = & 1 - (1-p)^{k_pt}(1-q)^{k_qt} - b (kt) (1-a) (1-p)^{k_pt}(1-q)^{k_qt}
\end{eqnarray*}
where
\[ b = \frac{\frac{pk_p}{1-p}+ \frac{qk_q}{1-q}}{k} \geq \frac{1}{k} \left( pk_p + qk_q \right) \qquad \mbox{and} \qquad 
b \leq \frac{1}{k} \left( pk_p + qk_q \right) + 2\max\{p^2,q^2\} = \phi/k + 2\max\{p^2,q^2\}
\]
Let $
(1-z)= \left( (1-p)^{k_p}(1-q)^{k_q} \right)^{1/k}$
which implies 
$(1-z)^{kt} = (1-p)^{k_pt}(1-q)^{k_qt}$. Let $c=\ln(1-z)$.
\[
c=\frac{1}{k} \left(k_p \ln (1-p) + k_q \ln (1-q) \right) \geq -\frac{1}{k} \left( pk_p + qk_q + p^2k_p + q^2 k_q \right)
\geq - b - \max\{p^2,q^2\} \]
now
\begin{eqnarray*}
\frac{\underline{A}(t)}{(\zeta_1 + \zeta_2)} & = & 1 - (1-z)^{kt} - b (kt) (1-a) (1-z)^{kt}\\
\implies \frac{\underline{A}'(t)}{(\zeta_1 + \zeta_2) k}& = & - c(1-z)^{kt} - b(1-a)(1-z)^{kt} - b(kt)c(1-a)(1-z)^{kt} \\
\implies \frac{\underline{A}''(t)}{(\zeta_1 + \zeta_2) k^2} 
& = & - c^2(1-z)^{kt} - 2bc(1-a)(1-z)^{kt} - c^2b(kt)(1-a)(1-z)^{kt}
\end{eqnarray*}
Therefore
\begin{eqnarray}
\frac{\underline{A}''(t)}{-c (1-z)^{kt} (\zeta_1 + \zeta_2) k^2} &  = &  c + 2(1-a)b +
cb(kt)(1-a) \nonumber
\\
& \geq & c + 2(1-a)b +
\phi t c(1-a) + 2ckt(1-a)\max\{p^2,q^2\} \qquad \mbox{(since $ c<0$)} \nonumber\\
&\geq & 2(1-a)b + \left(\frac43 - \frac13 a\right)c + ck(1-a)\frac{\max\{p^2,q^2\}}{2\phi} \quad \mbox{(since $\phi t\leq\frac13, c<0$)} \nonumber \\
& \geq & \left(\frac23 - \frac{5a}{3}\right)b - \left(\frac43 - \frac13 a\right) \max\{p^2,q^2\} + ck(1-a)\frac{\max\{p^2,q^2\}}{2\phi} 
\label{eqn:0001}
\end{eqnarray}

Observe that all the terms involving $p^2,q^2$ are $o(b)$ and $\underline{A}'' \ge 0$ when $a \leq 2/5$ which is true when $\zeta_1\leq 2\zeta_2/3$. Therefore $A(t)$ is convex.
\end{proof}

\subsubsection{Proof of Lemma~\ref{lem:pirlemma}}
\label{proof:pirlemma}

\begin{nlemma}{\ref{lem:pirlemma}}
For any $p\geq q, x\geq 1$ and any $t\geq 4r$ with $\phi xt \leq 1/3$, we have $\pi_r( xt) \leq 3\left( \frac{4x}{3(1-p)} \right)^r \pi_r(t)$
and if $3x(1-p)>4$ then
$\pi_r(t) \leq 4\left(\frac{4}{3x(1-p)}\right)^r \pi_r(xt)$.
\end{nlemma}

\begin{proof}
Consider an $r' \geq r, x\geq 1$, and $xt$ is an integer $\leq 1/(3\phi)
$\begin{eqnarray}
& & \Pr[ \bin(k_ptx,p)+\bin(k_qtx,q) = r']  \nonumber \\
&& \qquad
=\sum_{i=0}^{r'} \left({{k_pxt} \choose i} p^i (1-p)^{xtk_p-i} \right)
\left( {{k_qxt} \choose {r' - i}} q^{r'-i} (1-q)^{xtk_q-r'+i} \right) \nonumber \\
& & \qquad \leq 
\sum_{i=0}^{r'} \left( \frac{(xk_pt)^i}{i!} p^i (1-p)^{xtk_p-i} \right)
\left( x^{r'-i} \frac{(xk_qt)^{r'-i}}{(r'-i)!} q^{r'-i} (1-q)^{xtk_q-r'+i} \right) \\
& & \qquad = \sum_{i=0}^{r'}\left(\frac{\left(\frac{xpk_pt}{1-p} \right)^i}{i!}\frac{\left(\frac{xqk_qt}{1-q}\right)^{r'-i}}{(r'-i)!} \right) (1-p)^{xtk_p} (1-q)^{xtk_q} \nonumber \\
& & \qquad \leq \sum_{i=0}^{r'} \left(\frac{\left((1-p)^{-1}xpk_pt\right)^i}{i!}\frac{\left((1-p)^{-1}xqk_qt\right)^{r'-i}}{(r'-i)!} \right)
(1-p)^{xtk_p} (1-q)^{xtk_q} \nonumber \\
& & \qquad = \frac{\left( \frac{\phi xt}{1-p} \right)^{r'}}{r'!} (1-p)^{xtk_p} (1-q)^{xtk_q}  \label{eqn:a01}
\end{eqnarray}
In the case $y\geq 4r$ note that ${y \choose r} \geq (y-r)^r/r! \geq (3y/4)^r/r!$.
Thus in the case $xt \geq 4r$ we get
\begin{eqnarray}
& & \Pr[ \bin(k_ptx,p)+\bin(k_qtx,q) = r]  \nonumber 
=\sum_{i=0}^{r} \left({{k_pxt} \choose i} p^i (1-p)^{xtk_p-i} \right)
\left( {{k_qxt} \choose {r - i}} q^{r-i} (1-q)^{xtk_q-r+i} \right) \nonumber \\
& & \qquad \geq 
\sum_{i=0}^{r} \left( \frac{(xk_pt)^i}{(\frac43)^i i!} p^i (1-p)^{xtk_p} \right)
\left( \frac{(xk_qt)^{r-i}}{(\frac43)^{r-i} (r-i)!} q^{r-i} (1-q)^{xtk_q} \right) \nonumber \\
& & \qquad = \frac{(\phi x t)^r}{(\frac43)^r r!} (1-p)^{xtk_p} (1-q)^{xtk_q} \label{eqn:a02}
\end{eqnarray}

Therefore when $3xt\phi \leq 1$, we have:

\begin{eqnarray}
\frac{(\phi x t)^r}{(\frac43)^r r!} (1-p)^{xtk_p} (1-q)^{xtk_q} & \leq &  
\Pr[ \bin(k_ptx,p)+\bin(k_qtx,q) = r]  \mbox{(Equation~\ref{eqn:a02}.)} \nonumber \\
& \leq & \pi_r(xt) \label{later1001}\\
& \leq & \sum_{r' \geq r} \Pr[ \bin(k_ptx,p)+\bin(k_qtx,q) = r'] \nonumber \\
& \leq & \sum_{r' \geq r} \frac{\left(\frac{\phi xt}{1-p}\right)^{r'}}{r'!} (1-p)^{xtk_p} (1-q)^{xtk_q}  \mbox{(Equation \ref{eqn:a01}.)} \nonumber \\
& \leq & \frac{\left(\frac{\phi xt}{1-p} \right)^{r}}{r!} (1-p)^{xtk_p} (1-q)^{xtk_q} \sum_{r' \geq r} \left(\frac{\phi xt}{1-p}\right)^{r'-r} \nonumber \\
& \leq & \left(\frac{3}{2-3p} \right) \frac{\left(\frac{\phi xt}{1-p}\right)^{r}}{r!} (1-p)^{xtk_p} (1-q)^{xtk_q} \nonumber 
\end{eqnarray}
Summarizing the above we get for $x\geq 1, xt \geq 4r$:
\begin{eqnarray}
\label{eqn:a03}
\frac{(\phi x t)^r}{(\frac43)^r r!} (1-p)^{xtk_p} (1-q)^{xtk_q} \leq \pi_r(xt) \leq \left(\frac{3}{2-3p}\right) \frac{\left(\frac{\phi xt}{1-p}\right)^{r}}{r!} (1-p)^{xtk_p} (1-q)^{xtk_q}
\end{eqnarray}

But this immediately implies that (using Equation~\ref{later1001} in the second part):
\[ \pi_r(xt) \leq \left(\frac{3}{2-3p}\right) \left( \frac{4x}{3(1-p)} \right)^r \frac{(\phi t)^{r}}{(\frac43)^r r!} (1-p)^{tk_p} (1-q)^{tk_q} \leq 2\left( \frac{4x}{3(1-p)} \right)^r \pi_r(t) \]
and if $3x(1-p) > 4$,
\begin{eqnarray*} 
\pi_r(t) & \leq&  \left(\frac{3}{2-3p} \right) \frac{\left(\frac{\phi t}{1-p}\right)^{r}}{r!} (1-p)^{tk_p} (1-q)^{tk_q}  \leq 2\left( \frac{4}{3x(1-p)} \right)^r \frac{(\phi xt)^{r}}{(\frac43)^r r!} (1-p)^{tk_p} (1-q)^{tk_q} \\
& \leq & 2\left( \frac{4}{3x(1-p)} \right)^r \pi_r(xt)
(1-p)^{-(x-1)tk_p} (1-q)^{-(x-1)tk_q} \\
& \leq  & 2\left( \frac{4}{3x(1-p)} \right)^r \pi_r(xt)e^{2(x-1)tpk_p} e^{2(x-1)tqk_q} \qquad \mbox{(Since $e^{2y} \geq \frac1{1-y}$ for $y \in [0,\frac12]$.)} \\
& \leq & 2\left( \frac{4}{3x(1-p)} \right)^r \pi_r(xt) e^{2(x-1)\phi t} \leq 3\left( \frac{4}{3x(1-p)} \right)^r e^{2/3} \pi_r(xt) \\
& \leq & 4  \left( \frac{4}{3x(1-p)} \right)^r \pi_r(xt) 
\end{eqnarray*}
\noindent The lemma follows.
\end{proof}

\subsubsection{Proof of Lemma~\ref{lem:StSclose}}
\label{proof:StSclose}
\begin{lemma}
For any fixed $r>0$, define the stochastic processes
\begin{equation*}
\xi(t) = \frac{S_r^i(t) - \E [S_r^i(t)]}{1 - \pi_r(t)} \qquad \xi_{rev}(t) = \frac{S_r^i(t) - \E [S_r^i(t)]}{\pi_r(t)}
\end{equation*}
$\xi(t)$ is a martingale (i.e., $\E [\xi(t+1)] = \xi(t)$ for all $t$)
and $\xi_{rev}$ is a reverse martingale
(i.e., $\E [\xi_{rev}(t-1)] = \xi_{rev}(t)$ for all $t$)
.
\label{lem:martingale}
\end{lemma}

\begin{proof}
Let $V^i_r$ be the set of vertices in cluster $i$ with threshold $r$ that were not seeded. For $u \in V^i_r$, let $Y_u$ be the time at which $u$ becomes infected (set $Y_u = \infty$ if $u$ never becomes infected). Then $\pi_r(t) = \Pr[Y_u \leq t]$ and $S_r^i(t) = \sum_u \mathbbm{1}[ Y_u \leq t]$ where $\mathbbm{1}[]$ is the indicator function. The 
terms $\mathbbm{1}[ Y_u \leq t]$ are independent and identically distributed, so it suffices to show $\xi_u$ is a martingale, where

\begin{equation}
\xi_u(t) = \frac{ \mathbbm{1}[Y_u \leq t] - \Pr[Y_u \leq t]}{1 - \Pr[Y_u \leq t]} = 1 - \frac{ \mathbbm{1}[Y_u > t]}{1 - \Pr[Y_u \leq t]}
= 1 - \frac{\mathbbm{1}[Y_u > t]}{1 - \pi_r(t)}
\end{equation}

If $Y_u \leq t$, then $\xi_u(t) =\xi_u(t+1) =1$ and $\xi_u$ is a martingale. If $Y_u > t$,

\begin{equation*}
\xi_u(t+1) = \left\{ \begin{array}{ll}
\frac{-\pi_r(t+1)}{1 - \pi_r(t+1)} &\mbox{ if } Y_u > t+1 \qquad \Pr[Y_u > t+1 \mid Y_u > t] = \frac{1 - \pi_r(t+1)}{1 - \pi_r(t)} \\
1 &\mbox{ if } Y_u = t+1 \qquad \Pr[Y_u = t+1 \mid Y_u > t] = \frac{\pi_r(t+1) - \pi_r(t)}{1 - \pi_r(t)}
\end{array}\right.
\end{equation*}
and $\E [\xi_u(t+1)] = \xi_u(t)$, so $\xi_u$ is a martingale. Summing over all $u$, $ \xi(t)= \sum_u  \xi_u(t)$ is also a martingale. 
To show $\xi_{rev}$ is a reverse martingale, it suffices to show $\xi_v$ is reverse martingale, where

\newcommand{\xirev}{\xi_{rev}}

\begin{equation}
\xi_v(t) = \frac{ \mathbbm{1}[Y_u \leq t] - \Pr[Y_u \leq t]}{\Pr[Y_u \leq t]} =\frac{ \mathbbm{1}[Y_u \leq t]}{\Pr[Y_u \leq t]} - 1
= \frac{\mathbbm{1}[Y_u \leq t] - \pi_r(t)}{\pi_r(t)}
\end{equation}
If $Y_u > t$, then $\xi_v(t) = \xi_v(t-1) = -1$ and $\xi_v$ is a reverse martingale. If $Y_u \leq t$.
\begin{equation*}
\xi_v(t-1) = \left\{ \begin{array}{ll} \frac{1-\pi_r(t-1)}{\pi_r(t-1)} &\mbox{ if } Y_u  \leq t-1 \qquad \Pr[Y_u \leq t-1 \mid Y_u \leq t] = \frac{\pi_r(t-1)}{\pi_r(t)} \\
-1 &\mbox{ if } Y_u = t \qquad \Pr[Y_u = t \mid Y_u \leq t] = \frac{\pi_r(t) - \pi_r(t-1)}{\pi_r(t)}
\end{array} \right.
\end{equation*}
and $\E [\xi_v(t-1) ]= \xi_v(t)$, so $\xi_v$ is a reverse martingale and so is $\sum_v \xi_v$.
\end{proof}

\begin{lemma}
For any fixed $r>0$ and any $t_0$, 
\begin{equation*}
\E \left[ \left(\sup_{0< t \leq t_0} | S^i_r(t) - \E[ S^i_r(t)]| \right)^2 \right] \leq 16 \eta \zeta_r \pi_r(t_0) 
\end{equation*}
\label{lem:SESclose}
\end{lemma}

\begin{proof}
Let $\eta^i_r,\seed^i_r$ be the number of vertices in cluster $i$ with threshold $r$ and the vertices (with threshold $r$) which were seeded. 
Note that for a fixed $\eta^i_r,\seed^i_r$, $\Var_{\mbox{\tiny fixed } \eta^i_r,\seed^i_r} [S_r^i(t)] = (\eta_r - \seed^i_r) \pi_r(t) (1 - \pi_r(t)) < \eta^i_r \pi_r(t)$. Therefore taking the expectation over $\eta_r$ we get $\Var [S_r^i(t)] \leq \eta \zeta_r \pi_r(t)$.

Define $\xi$ and $\xi_{rev}$ to be the martingales from Lemma~\ref{lem:martingale}.
We will be using Doob's $L^p$ maximal inequality which states that for a Martingale $M(t)$, $p>1$ and any $\tau \geq 1$, 
\[ \E\left[  \left(\sup_{t \leq \tau} | M(t) | \right)^p \right] \leq \left(\frac{p}{p-1}\right)^p \E\left[  | M(\tau) |^p \right]
\]

We will use the inequality for $p=2$. We break the proof into two cases.

\textbf{Case I $\pi_r(t_0) \leq 1/2$, $t \leq t_0$.} We apply Doob's Inequality on $\xi$ with $\tau=t_0$ to get
\begin{equation}
\E \left[ \left(\sup_{t \leq t_0} | S^i_r(t) - \E [S^i_r(t)] | \right)^2
\right] \leq \E [ \sup_{t \leq t_0} | \xi(t)|^2] \leq 4 \E [| \xi(t_0) |^2 ]= 4 \frac{Var [S_r^i(t_0)]}{(1-\pi_r(t_0))^2} \leq 8 \eta \zeta_r \pi_r(t_0)
\label{eqn:mart1}
\end{equation}

\textbf{Case II $\pi_r(t_0) > 1/2$, $t \leq t_0$.} Observe that $\pi_r(0)=0$ and $\pi_r(t)$ is monotonic nondecreasing. Let $t_1$ be the largest integer such that $\pi_r(t_1) \leq 1/2$. Then using exactly the same argument as in Equation~\ref{eqn:mart1} we have

\begin{equation}
\E \left[ \left(\sup_{t \leq t_1} | S^i_r(t) - \E [S^i_r(t)] | \right)^2 \right] \leq 8 \eta \zeta_r \pi_r(t_1) \leq 8 \eta \zeta_r \pi_r(t_0)
\label{eqn:mart11}
\end{equation}

\noindent 
We use $\xi_{rev}$ to get
\begin{equation}
\E \left[ \left(\sup_{t \geq t_1+1} | S^i_r(t) - \E [S^i_r(t)]| \right)^2 \right] \leq 4 \frac{\Var[S^i_r(t_1+1)]}{\pi_r(t_1+1)^2} \leq 8 \eta\zeta_r (1 - \pi_r(t_1+1))\pi_r(t_1+1) \leq 8 \eta \zeta_r \pi_r(t_0)
\label{eqn:mart2}
\end{equation}

\noindent
Since 
\begin{equation*}
\E \left[ \left(\sup_{t \geq 0} | S^i_r(t) - \E [S^i_r(t)]| \right)^2 \right] \leq
\E \left[ \left(\sup_{t \leq t_1} | S^i_r(t) - \E[ S^i_r(t)]| \right)^2 \right] + 
\E \left[ \left(\sup_{t \geq t_1 + 1} | S^i_r(t) - \E[S^i_r(t)]| \right)^2 \right]
\end{equation*}
\noindent
And we apply Equations~\ref{eqn:mart11} and \ref{eqn:mart2} to get that
\begin{equation}
\E \left[ \left(\sup_{t \geq 0} | S^i_r(t) - \E [S^i_r(t)]| \right)^2 \right] \leq 8 \eta \zeta_r \pi_r(t_0) + 8 \eta \zeta_r (1 - \pi_r(t_0)) < 8 \eta \zeta_r < 16 \pi_r(t_0) \eta \zeta_r
\label{eqn:mart4}
\end{equation}

\noindent 
The lemma follows.
\end{proof}

\noindent We can now prove the main result of this subsubsection.

\begin{nlemma}{\ref{lem:StSclose}}
Let $t_0=\frac{96}{(1-p)^2\beta^2} \tc$ and suppose $\tc\rightarrow \infty$ as 
$n\rightarrow\infty$.
For $q\leq p\leq1/2$ and all fixed $\gamma,\beta>0$ if $n \geq n(\beta,\gamma,k)$ which is  sufficiently large, with probability at least $1-c_0/(\gamma^2 \beta^{2\rmax} t^*)$ for some
absolute constant $c_0$, simultaneously for all $i$,
\begin{equation*}
\sup_{1 \leq t \leq t_0} \left | S^i( t) - \left(\eta - \frac{\seed}{k} \right) A(t) \right| \leq \gamma  \tc
\end{equation*}
\end{nlemma}

\begin{proof}
Applying Lemma~\ref{lem:SESclose} 
and Lemma~\ref{lem:pirlemma} we get:

\begin{equation}
\E \left[\sup_{1 \leq t \leq t_0} \left | S_r^i(t) - \E \left[ S_r^i(t)\right] \right |^2 \right] = 16 \eta \zeta_r \pi_r(t_0) \leq 48 \left(\frac{128}{(1-p)^3\beta^2}  \right)^r \eta\zeta_r\pi_r(\tc)) \label{secondeqn}
\end{equation}

\noindent We then use the triangle inequality over $r = 1 \dots \rmax$, 
\begin{eqnarray*}
\E \left[ \sup_{1 \leq t \leq t_0} \left | S^i(t) - \E \left[ S^i(t)\right] \right |^2 \right] \leq 48 \left(\frac{128}{(1-p)^3\beta^2}  \right)^{\rmax} \sum_{r=1}^{\rmax} \eta\zeta_r\pi_r(\tc) \leq 48 \left(\frac{128}{(1-p)^3\beta^2}  \right)^{\rmax} \frac{n}{k} A(\tc) 
\end{eqnarray*}
\noindent But $nA(\tc) \leq \tc$ (Definition~\ref{def:maindef}) and $\tc \rightarrow \infty$, thus for sufficiently large $n$,
therefore for $\delta=c_0/(\gamma^2\beta^{2\rmax}\tc)$:
\begin{eqnarray*}
\E \left[ \sup_{1 \leq t \leq t_0} \left | S^i(t) - \E \left[ S^i(t)\right] \right |^2 \right] 
\leq  48 \left(\frac{128}{(1-p)^3\beta^2}  \right)^{\rmax} \frac{\tc}{k} 
\leq \frac{\delta \gamma^2 (\tc)^2}{2k}
\end{eqnarray*}

\noindent
Therefore using Markov inequality, for every $i$ with probability $1-\delta/(2k)$ we have
\[ \sup_{0 \leq t \leq t_0} \left | S^i(t) - \E \left[ S^i(t)\right] \right |^2 \leq \gamma^2 (\tc)^2 \]
\noindent
The lemma follow from the union bound over all $i$ and taking the square root.
\end{proof}

\subsubsection{Proof of Lemma~\ref{lem:pastbottleneck1}}
\label{proof:pb1}
\begin{nlemma}{\ref{lem:pastbottleneck1}}
When the seed size is $\overline{\seed}=(1+\epsilon)\pthreshold$, $\epsilon\leq 1/9$ and we have reached $t=\min\{\frac{96}{(1-p)^2\beta^2}\tc,1/(3\phi)\}$ then with
probability $1-O(\frac1{\tc})$, $S^i(t) > t-a$ for all $t \in [\frac{96}{(1-p)^2\beta^2}\tc, 1/(3\phi)]$, i.e., as $n$ increases and $\tc \rightarrow \infty$ the percolation continues till $t=1/(3\phi)$.
\end{nlemma}

\begin{proof}
Let $F^i(t) = \sum_{r=1}^{\rmax} \left(\bin(\eta_r-\overline{\seed}^i_r,
\pi_r(t)) + \overline{\seed}^i_r \right)$. Recall that $\eta_r$ is the number of
vertices with threshold $r$ and $\seed^i_r$ are the number of seeded
vertices with threshold $r$ in cluster $i$. Note that $\seed=\overline{\seed}$. 
We will show $F^i(t) >
(1+\epsilon)t$ for all $t$ with probability $1-O(\frac1{\tc})$.
Since the percolation has proceeded to at least $2\tc/\beta$ observe that
\[ (n - \overline{\seed})A(2\tc/\beta) + \overline{\seed} \geq 2k\tc/\beta \]
Expanding $A$
\[ (n- \overline{\seed})\zeta_1 \pi_1(2\tc/\beta) + (n - \overline{\seed}) \sum_{r=2}^{\rmax} \zeta_r \pi_r(2\tc/\beta) \geq 2k\tc/\beta - \overline{\seed}
\]
Now for $t\geq 1$, we have $\pi_1(t) = 1 - (1-p)^{k_pt}(1-q)^{k_qt} \leq 1 - (1 - pk_pt)(1-qk_qt) \leq pk_pt + qk_qt = t\phi$ (note $(1-p)^z \geq 1 - pz$ for all $z\geq 1$). 
Moreover $\zeta_1 \phi\eta =\zeta_1 \phi n/k \leq (1-\beta)$,
\begin{eqnarray*}
2k\tc/\beta - \overline{\seed}  & \leq & n \zeta_1 \phi 2\tc/\beta +  (n-\overline{\seed}) \sum_{r=2}^{\rmax} \zeta_r \pi_r(2\tc/\beta) \\
& \leq & (1-\beta)2k\tc/\beta +  (n-\overline{\seed}) \sum_{r=2}^{\rmax} \zeta_r \pi_r(2\tc/\beta)
\end{eqnarray*}
But $\pthreshold \leq k\tc$ (Definition~\ref{def:maindef}) and $\overline{\seed} \leq (1+\epsilon)k\tc \leq 10k\tc/9$ therefore 
\begin{equation}
\label{later003} 
(n-\overline{\seed}) \sum_{r=2}^{\rmax} \zeta_r \pi_r(2\tc/\beta) \geq 8k\tc/9 
\end{equation}
\noindent 
Let $t_j = 2z\tc/\beta$. Note $z \geq \frac{48}{(1-p)\beta}$. We use Lemma~\ref{lem:pirlemma},
\begin{eqnarray*}
\E[F^i(t_j)] & \geq & 
\sum_{r=1}^{\rmax} (\eta - \overline{\seed}/k) \zeta_r \pi_r(t_j) \geq  \frac1k(n - \overline{\seed}) \sum_{r=2}^{\rmax} \zeta_r \pi_r(t_j) \\
& \geq & \frac{9(1-p)^2z^2}{64} \frac1k(n - \overline{\seed}) 
\sum_{r=2}^{\rmax} \zeta_r \pi_r(2\tc/\beta) \qquad \mbox{(Using Lemma~\ref{lem:pirlemma})}\\
& \geq & \frac{9 (1-p)^2 z^2}{64} \frac{8\tc}{9} \qquad \mbox{(Using Equation~\ref{later003}.)}\geq 3t_j
\end{eqnarray*}

\noindent
We apply Chebyshev's inequality, note $\Var[ F^i(t)] \leq \E[F^i(t)]$.
\begin{equation*}
\Pr[F(t_j) \leq 2t_j] \leq \Pr \left[ F(t_j) \leq \frac23 \mathbb{E}[F(t_j)] \right] 
\leq \frac{\Var[F(t_j)]}{((1/3) E[F(t_j)])^2} \leq \frac{9}{\E[ F(t_j)]} 
\leq 3/t_j
\end{equation*}

\noindent
We now use union bounds over the intervals -- note that the 
sum of the probabilities of stopping in each interval telescopes to a total of $O(1/\tc)$.
The lemma follows.
\end{proof}

\subsubsection{Proof of Lemma~\ref{lem:closer}}
\label{proof:closer}
\begin{nlemma}{\ref{lem:closer}}
If the percolation has continued till $t=(3\phi)^{-1}$ then with
probability $1-1/n^{\Omega(1)}$, the percolation does not stop till a constant fraction of the graph is infected. Moreover if the expected degree $\phi \eta$ is a slowly growing function then with probability $1-1/n^{\Omega(1)}$, the percolation does not stop till $\eta -o(\eta)$ nodes are infected.
\end{nlemma}

\begin{proof}
The analysis corresponds to several different intervals.

\begin{statement}
For all $t \in [(3\phi)^{-1}, c_2 \eta]$, $S^i(t) > c_2 \eta > t$, where $c_2$ is some small constant.
\end{statement}

We can pessimistically pretend that every vertex has been assigned
threshold $\rmax$, and the probability that a vertex is healthy is at
most $\pi_{\rmax}(t) \geq \pi_{\rmax}(\underline{t})$ when
$\underline{t}=1/(3\phi)$.  Consider  $t \geq 4r_m$ and since $p,q \leq 1/2$,
\begin{eqnarray*}
\pi_{\rmax}(\underline{t}) & \geq & 
 \sum_{i=0}^{\rmax} {k_p \underline{t} \choose i} {k_q \underline{t} \choose {\rmax-i}} p^iq^j (1-p)^{k_p \underline{t} - i}(1-q)^{k_q \underline{t} - \rmax+i} \\
& \geq & \sum_{i=0}^{\rmax} \frac1{i!}\frac1{(\rmax-i)!} \left( \frac{3k_p \underline{t}}{4}\right)^i 
\left( \frac{3k_q \underline{t}}{4}\right)^{\rmax-i} p^i q^i e^{-2p k_p \underline{t}}e^{-2q k_q \underline{t}} = \left( \frac{3\phi \underline{t}}{4}\right)^{\rmax}\frac1{\rmax!} e^{-2\phi\underline{t}} =c_1
\end{eqnarray*}

\noindent for some constant $c_1 > 0$. This implies that $\E[\bin(\eta-\seed,
\pi_{\rmax}(t)]+\seed > c_1 \eta$. An application of Chernoff bounds
provides a $c_2$ such that $\bin(\eta, c_1) > c_2 \eta$ with
probability $1 - 1/n^{\Omega(1)}$. This implies that once
$(3\phi)^{-1}$ nodes are infected, in the very next generation a
constant fraction of the cluster is infected.
In the remainder we will try to bound the number of nodes who {\bf do
not get infected} and show that the number is small. This part is
identical to the proof in \cite{gnp}; because the analysis now
switches to the vertices who continue to survive -- and that analysis
does not depend on the threshold of infection. The first part of that 
analysis is:

\begin{statement}
For all $t \in [c_2 \eta, \eta - c_4/\phi)]$, for some constant $c_3$ the percolation does not stop with probability $1 - 1/\eta^{\Omega(1)}$.
\end{statement}

\noindent The probability of being healthy is bounded above
$1 - \pi_{\rmax}$, irresprective of the threshold. 
\begin{eqnarray}
1-\pi_{\rmax}(t) & = & 
\sum_{r=0}^{\rmax-1} \sum_{i=0}^{r} {k_pt \choose i} {k_qt \choose {r-i}} p^iq^j (1-p)^{k_pt - i}(1-q)^{k_qt - r+i} \nonumber \\
& \leq & \sum_{r=0}^{\rmax-1} \sum_{i=0}^r \frac1{i!}\frac1{(r-i)!} \left(k_pt\right)^i 
\left(k_qt\right)^{r-i} p^i q^i e^{-p k_pt +ip}e^{-q k_qt +(r-i)q} \nonumber \\
& \leq & e^{\rmax(p+q)} \sum_{r=0}^{\rmax-1} 
\left(\phi t \right)^{r} \frac1{r!} e^{-\phi t} 
=  \frac{ e^{\rmax(p+q)} e^{-\phi t}}{\phi t} \sum_{r=0}^{\rmax-1} 
\left(\phi t \right)^{r+1} r \frac1{(r+1)!} \label{later008}\\
& \leq & \frac{ \rmax  e^{\rmax(p+q)} e^{-\phi t}}{\phi t} e^{\phi t} = 
\frac{ \rmax  e^{\rmax(p+q)} }{\phi t} \leq 
\frac{ \rmax  e^{\rmax(p+q)} }{c_2 \phi \eta} = \frac{c_3}{\phi \eta} \nonumber
\end{eqnarray}
for some absolute constant $c_3$. Now the expected number of remaining
healthy vertices is $(\eta - t)(1-\pi_{\rmax}(t))$ which is $c_3/\phi$
in expectation. Therefore we can again apply Chernoff bound to prove that the percolation proceeds to $n - c_4/\phi$ nodes with probability $1- 1/\eta^{\Omega(1)}$ (Note $c_4 > c_3$.)

\begin{statement}
If $\phi\eta$ is a slowly growing function $\gg \rmax$ then the
percolation does not stop in the range for $t \in [\eta - c_4/\phi,
\eta - o(\eta)]$ with probability $1 - 1/\eta^{\Omega(1)}$.
\end{statement}

\noindent We reuse Equation~\ref{later008} once more and get
\[ 1- \pi_{\rmax}(t) \leq  e^{\rmax(p+q)} \sum_{r=0}^{\rmax-1} \frac{(\phi t)^r}{r!} e^{- \phi t} \leq  e^{\rmax(p+q)} \sum_{r=0}^{\rmax-1} \frac{(\eta\phi - c_4)^r}{r!} e^{- \eta \phi + c_4} = o(1)\]

\noindent The lemma follows.
\end{proof}

\section{Intervention Strategies}
\label{sec:intervention}
Recall from Definition~\ref{def:tmgraph} that for a vertex $u$ in a Templated Multisection graph, $v$ is near $u$ if $u$ and $v$ are connected with probability $p$ and $v$ is far from $u$ if $u$ and $v$ are connected with probability $q$.

\begin{definition}[Set of Healthy Vertices]
For a fixed generation $\tau$, we define $\HH$ to be the set of healthy vertices and define $\HH(r)$ to be the set of healthy vertices with threshold $r$. We define $\HH_a$ to be the set of healthy vertices with exactly $a$ infected neighbors. When $G$ is a Templated Multisection graph, we define $\HH_{b,c}$ to be the set of healthy vertices with exactly $b$ near infected neighbors and exactly $c$ far infected neighbors.
\end{definition}

\begin{definition}[Set of Infected Vertices]
We define $\I(\tau)$ to be the set of infected vertices at generation $\tau$ and define $\I(\tau-1)$ similarly.
\end{definition}

\begin{lemma}[Proved in Section~\ref{proof:HH}]
We can calculate $\Pr[u \in \HH_a \mid u \in \HH, r(u)=r]$ and $\Pr[u \in \HH_{b,c} \mid u \in \HH, r(u)=r]$ using $\I(\tau)$ and $\I(\tau-1)$. Additionally, if $\I(\tau) < k/(3 \phi)$, then $\Pr[u \in \HH_{a+1}] < (2/3) \Pr[u \in \HH_a]$.
\label{lem:HH}
\end{lemma}

\begin{ntheorem}{\ref{thm:interventionsuperthm}}
Assume $n,p,q,r,\lambda$, $\I(\tau)$, $\I(\tau-1)$, $\HH(r)$ are known and that $\I(\tau) < k/(3 \phi)$. Given either $\zeta(r)$ (for \bolster ), $z'$ (for \delay ), or $\alpha_p$ and $\alpha_q$ (for \diminish\ and \sequester ), it is possible to determine whether the intervention is successful and $G$ becomes infected with probability $1-1/poly(n)$.
\end{ntheorem}
\par

\medskip
Note that it is virtually impossible to apply an intervention when exactly $\lambda n$ vertices are infected, for example the case where $\I(10)=.95 \lambda n$ and $\I(11)=1.05 \lambda n$. As a result, small changes in $\lambda$ may have no impact on whether the intervention is successful or not; the true determining factor is the size of $\I(\tau)$ and $\I(\tau+1)$.

Also consider two different graphs $G_1$ and $G_2$ and let $\I_i(t)$ denote the number of infected vertices in $G_i$ at time $t$. If $|\I_1(\tau)| < |\I_2(\tau)|$ and $|\I_1(\tau+1)-\I_1(\tau)| < |\I_2(\tau+1)-\I_2(\tau)|$ then $|\I_1(\tau+1)| < |\I_2(\tau+1)|$ and every intervention that successfully stops the percolation on $G_2$ will also stop the percolation on $G_1$. However if $|\I_1(\tau+1)| < |\I_2(\tau+1)|$ but $|\I_1(\tau+1)-\I_1(\tau)| > |\I_2(\tau+1) - \I_2(\tau)|$, there is no guarantee that an intervention that stops the percolation on $G_1$ stops the intervention on $G_2$. See also the discussing about Figure~\ref{fig:bolster3d}.

\subsection{\bolster\ and \delay\ Intervention}
\label{sec:bolster}

We begin by formally defining \bolster\ intervention.

\begin{definition}[\bolster\ intervention]
Define the intervention generation $\tau$ to be $\tau = \min_t |\I(t+1)| > \lambda n$ and for every $r$, let $\zeta'(r)$ be a distribution on $[r, \dots, r'_m]$. Every non-infected vertex with threshold $r$ will be assigned a new threshold from distribution $\zeta'(r)$.

For the first $\tau-1$ generations, we run the standard bootstrap percolation process. Then the sequence of events are:

\begin{enumerate}\parskip=0in
\item Generation $\tau$ begins. Every vertex counts its infected neighbors. Every vertex with $r(u)$ or more infected neighbors becomes infected. Note that $|\I(\tau)| < \lambda n$.
\item Generation $\tau+1$ begins. Every vertex counts its infected neighbors. Every vertex $u$ with $r(u)$ or more infected neighbors becomes infected.
\item Every vertex $u$ with less than $r(u)$ infected vertices is assigned a new threshold from distribution $\zeta'(r(u))$. Note that $|\I(\tau+1)| > \lambda n$. 
\end{enumerate}
\end{definition}

\begin{definition}[\delay\ intervention]
Delay intervention is a special case of \bolster\ intervention where $\zeta'_j(r) = (1-z'(u))^{j-r} z'(u)$.
\end{definition}

Note that with this definition \bolster\ intervention cannot save a vertex that is about to become infected; if $u$ has $r(u)$ infected neighbors when the intervention is applied it still becomes infected. Our results focus on the definition above, at the end of the section we describe how to modify the results using either of the following alternate definitions.

\begin{modification}
\bolster\ intervention is allowed to save vertices. In the definition above, when generation $\tau+1$ begins, every vertex is assigned a new threshold from $\zeta'$ before checking whether vertices become infected; essentially we swap steps 2 and 3 in the definition.
\label{mod1}
\end{modification}

Modification~\ref{mod1} is substantially stronger than our definition of \bolster\, it is in fact so strong that it is difficult to generate interesting simulation data (even the weakest interventions are always successful).

\begin{modification}
\bolster\ intervention is allowed to weaken vertices, and $\zeta'(r)$ is allowed to be a distribution on $[2, r'_m]$ instead of $[r, r'_m]$.
\label{mod2}
\end{modification}

If $\zeta'(r)$ is any distribution on $[r, r'_m]$, our analysis holds. If $\zeta'(r)$ is allowed to be a distribution on $[2, r'_m]$, then badly chosen $\zeta'$ may lead to problems (as discussed in the end of the section). One example of a badly chosen intervention is $\zeta'(r) = r-1$, the `intervention' that reduces every vertex's threshold by 1.

We now begin the proof of Theorem~\ref{thm:interventionsuperthm} for \bolster\ , i.e. determining whether \bolster\ stops the spread of percolation. We construct a new graph $J$ of the same 'type` as $G$ but with $|\HH|$ vertices. We will choose thresholds for the vertices of $J$ so that the probability $J$ becomes infected is equal to the probability $G$ becomes infected. For every $u \in \HH_a$, let $r'(u)$ denote the new threshold of $u$. $u$ becomes infected when it has $r'(u)-a$ infected neighbors in $\HH$ to go along with its $a$ infected neighbors in $\I(\tau)$. Thus, we will add a vertex $v$ to $J$ with threshold $r'(u)-a$. $v$ becomes infected when it has $r'(u)-a$ neighbors in $J$. In this way, we encode the information about $\I(\tau)$ into the thresholds of $J$, and the probability $u \in \HH$ becomes infected is equal to the probability that $v \in J$ becomes infected.

For example, $v$ has threshold $2$ if $r'(u) = 2$ and $u$ had zero infected neighbors  or $r'(u) = 3$ and $u$ had one infected neighbors -or- $r'(u)=4$ and $u$ had two infected neighbors and so on.

Formally, let $G = TM(F, n, k_p, k_q, p, q)$ and $J = (F, |\HH|, k_p, k_q, p, q)$. Recall we can calculate $\Pr[u \in \HH_a \mid r(u)=r]$ using Lemma~\ref{lem:HH}. We now define $j_s$ and $\seed$ as follows.

\begin{equation}
j_s = \sum_{r=2}^{r_m} \sum_{a=0}^{r} \frac{|\HH(r)|}{|\HH|} \Pr[u \in \HH_a \mid r(u)=r] * \mathbf{1}[a < r] * \zeta'_{s+a}(r) \qquad \seed = |\HH| - |\HH| \sum_{s=1}^{r'_m} j_s
\label{eqn:jsseed}
\end{equation}

We will let $j_1, j_2, \dots, j_{r'_m}$ be the distribution used to assign thresholds and $\seed$ will be the number of seed vertices.  We then use Theorem~\ref{thm:mainthm} to determine whether $J$ becomes infected. If $J$ becomes infected with polynomially high probability, then $G$ also becomes infected and with that same probability and the intervention is not successful. If $J$ does not become infected, then the intervention is successful.

In order to apply Theorem~\ref{thm:mainthm}, we need to confirm that $j_1 < (2/3) j_2$. Note that

\begin{eqnarray*}
j_1 &=& \sum_{r=2}^{r_m} \frac{|\HH(r)|}{|\HH|} \Pr[u \in \HH_{r-1} \mid r(u)=r] \zeta'_{r}(r) \\
j_2 &=& \sum_{r=2}^{r_m} \frac{|\HH(r)|}{|\HH|} \big( \Pr[u \in \HH_{r-2} \mid r(u)=r] \zeta'_{r}(r) +
\Pr[u \in \HH_{r-1} \mid r(u)=r] \zeta'_{r+1}(r) \big)
\end{eqnarray*}

By Lemma~\ref{lem:HH}, $\Pr[u \in \HH_{r-1}] < (2/3) \Pr[u \in \HH_{r-2}]$. This immediately implies that $j_1 < (2/3) j_2$ and that Theorem~\ref{thm:mainthm} applies.
To use Modification~\ref{mod1}, remove the $\textbf{1}[a<r]$ part of Equation~\ref{eqn:jsseed}. To use Modification~\ref{mod2}, no changes to Equation~\ref{eqn:jsseed} are necessary. However, note that with Modification~\ref{mod2}, $j_1$ is not guaranteed to be less than $(2/3)j_2$; depending on the distribution $\zeta'$, it might be the case that $j_1 > (2/3) j_2$. At this point, Theorem~\ref{thm:mainthm} no longer applies.

\subsection{\diminish\  and \sequester\ Intervention}

We begin by formally defining \diminish\ intervention, which corresponds to the idea of deleting edges randomly.

\begin{definition}[\diminish\ intervention]
Define the intervention generation $\tau$ to be $\tau = \min_t |E(t)| > \lambda n$. Let $\zeta'$ be a distribution of new vertex thresholds. For the first $\tau-1$ generations, we run the standard bootstrap percolation process. Then the sequence of events are:

\begin{enumerate}\parskip=0in
\item Generation $\tau-1$ begins. Every vertex counts its infected neighbors. Every vertex with $r(u)$ or more infected neighbors becomes infected.
\item Generation $\tau$ begins. Delete that edge connecting two near vertices with probability $1-\alpha_p$. Delete every edge connecting two far vertices with probability $1-\alpha_q$. After edges are deleted, every vertex counts its infected neighbors. Every vertex with $r(u)$ infected neighbors in $G'$ becomes infected.
\end{enumerate}
\end{definition}

\begin{definition}[\sequester\ intervention]
\sequester\ intervention is defined similarly to \diminish\, but edges connecting two healthy vertices are never deleted. Edges connected to an infected vertices are deleted with probability $1-\alpha_p$ or $1-\alpha_q$.
\end{definition}

Note that unlike \bolster\  intervention, we do allow \diminish\ to `save' vertices about to be infected. 
Let $G'$ be the post-intervention graph after edges are deleted and define $\HH'_a$ to be the set of healthy vertices that have exactly $a$ infected vertices in $G'$. Define $\HH'_{b,c}$ similarly. When $G$ is a Erdos-Reyni graph, we get

\begin{eqnarray*}
\Pr[u \in \HH'_a \mid r(u)=r, u \in \HH] = \sum_{d=a}^\infty \Pr[\bin(d, \alpha)=a] \Pr[u \in \HH_a \mid r(u)=r, u \in \HH]
\end{eqnarray*}

When $G$ is a TM graph, we get

\begin{eqnarray*}
\Pr[u \in \HH'_{b,c} \mid r(u)=r, u \in \HH] = \sum_{d=b}^\infty \sum_{e=c}^\infty 
\Pr[\bin(d, \alpha_p) = b] \Pr[ \bin(e, \alpha_q) = c] \Pr[u \in \HH_{b,c} \mid r(u)=r, u \in \HH]
\end{eqnarray*}

and $\Pr[u \in \HH'_a] = \sum_{b+c=a} \Pr[u \in \HH'_{b,c}]$.
The remainder of the analysis is very similar to the \bolster\ case, but using $\HH'_a$ instead of $\HH_a$.  We construct a new graph $J$, and for every $u \in \HH'_a$, we add a vertex $v$ to $J$ with threshold $r(u)-a$. $J$ will be a TM graph with $\alpha_p p$ and $\alpha_q q$ edge probability instead of $p$ and $q$. Formally, let $G = TM(F, n, k_p, k_q, p, q)$ and $J = (F, |\HH|, k_p, k_q, \alpha_p p, \alpha_q q)$. 

\begin{equation*}
j_s = \sum_{r=2}^{r_m} \sum_{a=0}^{r} \frac{|\HH(r)|}{|\HH|} \Pr[u \in \HH'_a \mid r(u)=r] * \mathbf{1}[a < r]  \qquad \seed = |\HH| - |\HH| \sum_{s=1}^{\rmax} j_s
\end{equation*}

We then use Theorem~\ref{thm:mainthm} on $J$ to determine whether the intervention is successful or not. For \sequester\ intervention, we define $J = (F, |\HH|, k_p, k_q, p, q)$ but use the same distribution of $j_s$.

\subsection{Proof of Lemma~\ref{lem:HH}}
\label{proof:HH}

As a warp up, we first consider the easier case where $G$ is an \er\ graph and $r(u)$ is known.

\setcounter{statement}{0}

\begin{statement}
When $G$ is an \er\ graph, $\Pr[ u \in \HH_a \mid u \in \HH, r(u)=r]$ can be estimated using $\I(\tau)$ and $\I(\tau-1)$.
\end{statement}

\begin{proof}
We define $B_i(x)=\Pr[ \bin(x,p) = i]$. For any sets $S_1$ and $S_2$, define $\sigma(S_1, S_2)$ to be the number of edges connecting $S_1$ and $S_2$. We begin by breaking $\Pr[u \in HH_a]$ into its component parts. For the remainder of the proof, we will suppress the conditional $u \in \HH, r(u)=r$ in the interest of space; all probabilities are conditioned on knowing $u \in \HH$, and $r(u)=r$.

\begin{eqnarray*}
\Pr[u \in \HH_a \mid r(u)=r] &=& \Pr[ \sigma(u, \I(\tau)) = a] \\
&=& \sum_{d=0}^{r-1}  \Pr[ \sigma(u, \I(\tau-1)) = d]  \Pr[ \sigma(u, \I(\tau)-\I(\tau-1)) = a-d]
\end{eqnarray*}

If $u \in \HH$, $u$ cannot be connected to $r$ or more vertices in $\I(\tau-1)$; if $u$ was connected to $r$ or more vertices, $u$ would belong to $\I(\tau)$ and it would not belong to $\HH$. We can capture this constraint with a conditional binomial.

\begin{eqnarray*}
\Pr[ \sigma(u, \I(\tau -1))=d] = \frac{B_d(|\I(\tau -1)|)}{\sum\limits_{i=1}^{r-1} B_{i}(|\I(\tau -1)|)}
\end{eqnarray*}

In contrast, if $v \in \I(\tau)-\I(\tau-1)$, then $u$ and $v$ are connected with probability $p$ independent of the other edges, so

\begin{equation*}
\Pr[ \sigma(u, \I(\tau) - \I(\tau-1))=d-a] = B_{d-a}(|\I(\tau) - \I(\tau-1)|)
\end{equation*}

The statement follows from the last three equations. 
\end{proof}

We now consider the case where $G$ is a Templated Multisection graph. Conceptually, the ideas underlying the proof are identical to the \er\ graph.

\begin{statement}
When $G$ is a Templated Multisection graph, $\Pr[u \in \HH_a \mid u \in \HH, r(u)=r]$ and $\Pr[ u \in \HH_{b,c} \mid u \in \HH, r(u)=r]$ can be estimated using $\I(\tau)$ and $\I(\tau-1)$.
\end{statement}

\begin{proof}

We define $B_i(x)=\Pr[ \bin(x,p) = i]$ and $C_i(x) = \Pr[ \bin(x,q)=i]$ . Define $\sigma(S_1, S_2)$ to be the number of edges connecting $S_1$ and $S_2$. For a fixed vertex $u$, we break $\I(t)$ into its two component parts.

\begin{eqnarray*}
\I^{near}(\tau) = \I(\tau) \cap \{ v : v \mbox{ is near } u \} \qquad \I^{far}(\tau) = \I(\tau) \cap \{ v : v \mbox{ is far from } u \}
\end{eqnarray*}

and define similar expressions for $\tau-1$. We begin by breaking $\Pr[u \in \HH_{b,c}]$ into its component parts. For the remainder of the proof, we will suppress the conditional $r(u)=r$ in the interest of space; all probabilities are conditioned on knowing $r(u)=r$.

\begin{eqnarray*}
\Pr[ u \in H_{b,c} \mid r(u)=r] = \Pr[\sigma(u, \I^{close}(\tau)) = b \wedge \Pr[\sigma(u, \I^{far}(\tau)) = c]
\end{eqnarray*}

Note that these terms are not independent. 
\begin{eqnarray*}
&&\Pr[ \sigma(u, \I^{near}(\tau)) = b] = \sum_{d=0}^{b} \Pr[ \sigma(u,  \I^{near}(\tau -1)) = d] \Pr[ \sigma(u, \I^{near}(\tau) - \I^{near}(\tau -1)) = b-d]  \\
&&\Pr[ \sigma(u, \I^{far}(\tau)) = c] = \sum_{e=0}^{c} \Pr[ \sigma(u,  \I^{far}(\tau -1)) = e] \Pr[ \sigma(u, \I^{far}(\tau) - \I^{far}(\tau -1)) = c-e]  \\
\end{eqnarray*}

$\sigma(u, \I^{near}(\tau) - \I^{near}(\tau-1))$ is independent from the other terms, as is $\sigma(u, \I^{far}(\tau) - \I^{far}(\tau-1))$. Thus, we get

\begin{eqnarray*}
\Pr[u \in H_{b,c}] = \sum_{d=0}^b \sum_{e=0}^c \bigg( & \Pr[ \sigma(u, \I^{near}(\tau -1) = d \wedge \sigma(u, \I^{far}(\tau-1)) = e] \\
& *  \Pr[ \sigma(u, \I^{near}(\tau)-\I^{near}(\tau - 1))=b-d] \\
& * \Pr[ \sigma(u, \I^{far}(\tau) - \I^{far}(\tau-1)) = c-e] \bigg)
\end{eqnarray*}

\noindent Observe $\sigma(u, \I^{near}(\tau-1)) + \sigma(u, \I^{far}(\tau-1)) \leq r-1$. We can use a modified equation similar to the conditional binomial to get

\begin{eqnarray*}
\Pr[ \sigma(u, \I^{near}(\tau-1) = d) \wedge \sigma(u, \I^{far}(\tau-1) = e)] =
\frac{ B_d(|\I^{near}(\tau-1)|)*  C_e(|\I^{far}(\tau-1)|)}
{\sum\limits_{f+g \leq r-1}  B_f(|\I^{near}(\tau-1)|) * C_g(|\I^{far}(\tau-1)|)}
\end{eqnarray*}

\noindent 
Note that if $v \in \I^{near}(\tau)-\I^{near}(\tau-1)$, then $u$ and $v$ are connected with probability $p$ independent of all other edges (and similarly for $\I^{far}$ and $q$). Thus, 
\begin{eqnarray*}
 \Pr[ \sigma(u, \I^{near}(\tau) - \I^{near}(\tau-1)) = b-d] &=& B_{b-d}(|\I^{near}(\tau) - \I^{near}(\tau-1)|) \\
 \Pr[ \sigma(u, \I^{far}(\tau) - \I^{far}(\tau-1)) = c-e] &=& C_{c-e}(|\I^{far}(\tau) - \I^{far}(\tau-1)|) \\
\end{eqnarray*}

Combining the previous equations gives the formula for $\Pr[ u \in \HH_{b,c}]$. The first part of the statement follows from $\Pr[u \in \HH_a] = \sum_{b+c=a} \Pr[u \in \HH_{b,c}]$.
\end{proof}

We now begin working for the second part of Lemma~\ref{lem:HH}. The following lemma follows from facts about binomials.

\begin{lemma}
Define $B_r(t)=\Pr[ \bin(k_pt,p) =r]$, $C_r(t)=\Pr
[ \bin(k_qt,q) = r]$ and  $ D_r(t)=\Pr[ \bin(k_pt,p) + \bin(k_qt,q) = r]$. Then for $t > r$,
\begin{eqnarray*}
 B_{r+1}(t) &=& \frac{k_p t - r}{r+1} \frac{p}{1-p} B_r(t) \leq t k_p p (1-p)^{-1} B_r(t) \\ 
C_{r+1}(t) &<& \frac{k_q t - r}{r+1} \frac{q}{1-q} C_r(t)  \leq  t k_q q (1-q)^{-1} C_r(t) \\
D_{r+1}(t) & < &\phi t (1 - \max \{p,q \})^{-1} D_r(t)
\end{eqnarray*}
\label{lem:bcdhelper}
\end{lemma}

\begin{proof}
We begin with the proof for $B_{r+1}(t)$.
\begin{eqnarray*}
 B_{r+1}(t) &=& \binom{k_p t}{r+1} p^{r+1} (1-p)^{k_p t-r-1} = \frac{k_p t-r}{r+1} \binom{k_p t}{r} \frac{p}{1-p} p^{r} (1-p)^{k_p t-r} \\
& = &  \frac{k_p t-r}{r+1} \frac{p}{1-p} B_r(t) < t k_p p B_r(t)
\end{eqnarray*}
The proof of $C_{r+1}(t)$ follows similar logic. For $D_{r+1}(t)$, we get
\begin{eqnarray*}
 D_{r+1}(t) & =& \sum_{i=0}^{r+1} B_i(t) C_{r+1-i}(t) = \sum_{i=0}^r B_i(t) C_{r+1-i}(t) + B_{r+1}(t) C_0(t) \\
&\leq& \frac{t k_q q}{1-q} \sum_{i=0}^r B_i(t) C_{r-i}(t) + \frac{t k_p p}{1-p}  B_r(t) C_0(t)  \leq \frac{t k_q q}{1-q} \sum_{i=0}^r B_i(t) C_{r-i}(t) + \frac{t k_p p}{1-p} \sum_{i=0}^r B_i(t) C_{r-i}(t) \\
&& \leq \frac{\phi t}{1 - \max \{p, q\}} \sum_{i=0}^r B_i(t) C_{r-i}(t) = \frac{\phi t}{1-\max \{p, q\} } D_r(t)
\end{eqnarray*}
\end{proof}

We are now ready to prove the second half of Lemma~\ref{lem:HH}

\begin{statement}
If $\I(\tau) < k/(3 \phi)$, then $|\HH_{a+1}| < (2/3) |\HH_{a}|$.
\end{statement}

\begin{proof}
Let $t = \I(\tau)/k$. If $\I(\tau) < k/(3 \phi)$, then $t < 1/(3 \phi)$ and $\phi t < 1/3$. This implies $\phi t (1 - \max\{p, q\}) < 2/3$. For every $u \in \HH$, $\Pr[ u \in \HH_a] = D_{a}(t)$. 

\begin{eqnarray*}
\Pr[u \in \HH_{a+1}] = D_{a+1}(t) < \frac{\phi t}{1 - \max\{p,q\}} D_a(t) < \frac{2}{3} D_a(t) < \frac{2}{3} \Pr[u \in \HH_a]
\end{eqnarray*}
\end{proof}

If $\I(\tau) > k/(3 \phi)$ and $t = \I(\tau)/k$, then $t > 1/(3 \phi)$. At this point, Lemma~\ref{lem:closer}, Statement 1 applies and $\I(\tau+1)$ consists of a positive fraction of the nodes. 

\bibliographystyle{abbrv}
\bibliography{references}

\begin{thebibliography}{10}

\bibitem{functionofdegree}
H.~Amini.
\newblock Bootstrap percolation and diffusion in random graphs with given
  vertex degrees.
\newblock {\em Electronic Journal of Combinatorics}, 17(R25), 2010.

\bibitem{powerlaw}
H.~Amini and N.~Fountoulakis.
\newblock Bootstrap percolation in power-law random graphs.
\newblock {\em J. of Stat. Phy.}, 155(1):72--92, 2014.

\bibitem{hypercube}
J.~Balogh and B.~Bollob{\'a}s.
\newblock Bootstrap percolation on the hypercube.
\newblock {\em Probability Theory and Related Fields}, 134(4):624--648, 2006.

\bibitem{dgridsharp}
J.~Balogh, B.~Bollob{\'a}s, H.~Duminil-Copin, and R.~Morris.
\newblock The sharp threshold for bootstrap percolation in all dimensions.
\newblock {\em Transactions of the American Mathematical Society},
  364(5):2667--2701, 2012.

\bibitem{3gridsharp}
J.~Balogh, B.~Bollob{\'a}s, and R.~Morris.
\newblock Bootstrap percolation in three dimensions.
\newblock {\em The Annals of Probability}, pages 1329--1380, 2009.

\bibitem{regulartree}
M.~Biskup and R.~H. Schonmann.
\newblock Metastable behavior for bootstrap percolation on regular trees.
\newblock {\em Journal of Statistical Physics}, 136(4):667--676, 2009.

\bibitem{dregular}
L.~Blume, D.~Easley, J.~Kleinberg, R.~Kleinberg, and E.~Tardos.
\newblock Which networks are least susceptible to cascading failures?
\newblock In {\em Foundations of Computer Science}, pages 393--402, 2011.

\bibitem{nphard}
N.~Chen.
\newblock On the approximability of influence in social networks.
\newblock {\em SIAM Journal on Discrete Mathematics}, 23(3):1400--1415, 2009.

\bibitem{hier2}
A.~Clauset, C.~Moore, and M.~E.~J. Newman.
\newblock Structural inference of hierarchies in networks.
\newblock {\em Proc. ICML}, pages 1 -- 13, 2006.

\bibitem{hier}
A.~Clauset, C.~Moore, and M.~E.~J. Newman.
\newblock Hierarchical structure and the prediction of missing links in
  networks.
\newblock {\em Nature}, 453:98--101, 2008.

\bibitem{expanders}
A.~Coja-Oghlan, U.~Feige, M.~Krivelevich, and D.~Reichman.
\newblock {\em Contagious Sets in Expanders}, chapter 131, pages 1953--1987.
\newblock SIAM, 2013.

\bibitem{homogeneoustree}
L.~Fontes and R.~Schonmann.
\newblock Bootstrap percolation on homogeneous trees has 2 phase transitions.
\newblock {\em Journal of Statistical Physics}, 132(5):839--861, 2008.

\bibitem{torus}
J.~Gravner, C.~Hoffman, J.~Pfeiffer, D.~Sivakoff, et~al.
\newblock Bootstrap percolation on the hamming torus.
\newblock {\em The Annals of Applied Probability}, 25(1):287--323, 2015.

\bibitem{2gridshaper}
J.~Gravner, A.~E. Holroyd, and R.~Morris.
\newblock A sharper threshold for bootstrap percolation in two dimensions.
\newblock {\em Probability Theory and Related Fields}, 153(1-2):1--23, 2012.

\bibitem{2gridsharp}
A.~E. Holroyd.
\newblock Sharp metastability threshold for two-dimensional bootstrap
  percolation.
\newblock {\em Probability Theory and Related Fields}, 125(2):195--224, 2003.

\bibitem{gnp}
S.~Janson, T.~Luczak, T.~Turova, and T.~Vallier.
\newblock Bootstrap percolation on the random graph $g_{n,p}$.
\newblock {\em Ann. Appl. Probab.}, 22(5):1989--2047, 10 2012.

\bibitem{KDD2014}
E.~B. Khalil, B.~N. Dilkina, and L.~Song.
\newblock Scalable diffusion-aware optimization of network topology.
\newblock {\em Proc. KDD}, pages 1226--1235, 2014.

\bibitem{lecam}
L.~Le~Cam.
\newblock An approximation theorem for the poisson binomial distribution.
\newblock {\em Pacific Journal of Mathematics}, 10(4):1181–1197, 1960.

\bibitem{vaccination}
M.~Lelarge.
\newblock Efficient control of epidemics over random networks.
\newblock {\em ACM SIGMETRICS Performance Evaluation Review}, 37(1):1--12,
  2009.

\bibitem{kronecker}
J.~Leskovec and C.~Faloutsos.
\newblock Scalable modeling of real graphs using kronecker multiplication.
\newblock {\em Proc. ICML}, pages 497--504, 2007.

\bibitem{scalia}
G.~Scalia-Tomba.
\newblock Asymptotic final-size distribution for some chain-binomial processes.
\newblock {\em Advances in Applied Probability}, pages 477--495, 1985.

\bibitem{swarm}
J.~von Brecht, T.~Kolokolnikov, A.~Bertozzi, and H.~Sun.
\newblock Swarming on random graphs.
\newblock {\em Journal of Statistical Physics}, 151(1-2):150--173, 2013.

\bibitem{majorityrule}
D.~J. Watts.
\newblock A simple model of global cascades on random networks.
\newblock {\em Proceedings of the National Academy of Sciences},
  99(9):5766--5771, 2002.

\end{thebibliography}

\end{document}